\documentclass{amsart}
\usepackage{amsmath,mathrsfs,amssymb,amsthm,amscd}
\usepackage[]{fontenc}
\usepackage{xy}
\usepackage[hyperindex]{hyperref}
\hypersetup{breaklinks=true}
\usepackage{graphicx,color}
\usepackage{ifpdf}
\xyoption{all}

\numberwithin{equation}{section}
\nopagebreak%
\theoremstyle{plain}
\newtheorem{theorem}[equation]{Theorem}
\newtheorem{corollary}[equation]{Corollary}
\newtheorem{proposition}[equation]{Proposition}
\newtheorem{lemma}[equation]{Lemma}

\newtheorem{theorem-definition}[equation]{Theorem Definition}

\theoremstyle{definition}
\newtheorem{definition}[equation]{Definition}
\newtheorem{notation}[equation]{Notation}
\newtheorem{example}[equation]{Example}

\newtheorem{remark}[equation]{Remark}

 \DeclareMathOperator{\dv}{div}
\DeclareMathOperator{\diva}{\widehat{div}}
\DeclareMathOperator{\ocone}{\overline{cone}}

\DeclareMathOperator{\cha}{\widehat{CH}}
 \DeclareMathOperator{\CH}{CH}

\DeclareMathOperator{\za}{\widehat{Z}}
\DeclareMathOperator{\Zr}{Z}
\DeclareMathOperator{\rata}{\widehat{Rat}}
\DeclareMathOperator{\Td}{Td}
\DeclareMathOperator{\cht}{\widetilde{ch}}
\DeclareMathOperator{\chh}{\widehat{ch}}

 \DeclareMathOperator{\dd}{d}
\DeclareMathOperator{\Img}{Im}

\DeclareMathOperator{\ch}{ch} \DeclareMathOperator{\cl}{cl}

\DeclareMathOperator{\Ob}{Ob} 
 
 \DeclareMathOperator{\Ker}{Ker}
\DeclareMathOperator{\Hom}{Hom}

\DeclareMathOperator{\Id}{id}

\DeclareMathOperator{\Spec}{Spec}

\DeclareMathOperator{\amap}{a}
\DeclareMathOperator{\cmap}{c}
\DeclareMathOperator{\cone}{cone} 
\DeclareMathOperator{\bmap}{b}

\DeclareMathOperator{\KA}{\mathbf{KA}}

\DeclareMathOperator{\Db}{\mathbf{D}^{b}}

\DeclareMathOperator{\hDb}{\mathbf{\widehat{D}}^{b}}

\DeclareMathOperator{\oDb}{\overline{\mathbf{D}}^{b}}

\DeclareMathOperator{\Sm}{\mathbf{Sm}}
\DeclareMathOperator{\Reg}{\mathbf{Reg}}

\newcommand{\ov}{\overline}
\newcommand{\dra}{\dashrightarrow}

\newcommand{\ttwist}{\flat}

\newcommand{\ZZ}{{\mathbb Z}}
\newcommand{\RR}{{\mathbb R}}

\newcommand{\CC}{{\mathbb C}}
\newcommand{\QQ}{{\mathbb Q}}

\newcommand{\OO}{{\mathcal O}}

\newcommand{\PP}{{\mathbb P}}

\newcommand{\as}{{\text{\rm a}}}
\newcommand{\lllas}{{\text{\rm l, ll, a}}}

\newcommand{\D}{\text{{\rm cur}}}

\def\?{\ ???\ \immediate\write16{}%
\immediate\write16{Warning: There was still a question mark...}
\immediate\write16{}}

\begin{document}
\title{The arithmetic Grothendieck-Riemann-Roch theorem for general
  projective morphisms}

\author{Jos\'e Ignacio Burgos Gil} 
\address{Instituto de Ciencias Matem\'aticas (CSIC-UAM-UCM-UC3), 
Calle Nicol\'as Cabrera 15
28049 Madrid, Spain}
\email{burgos@icmat.es} 

\author{Gerard Freixas i Montplet}
\address{Institut de Math\'ematiques de Jussieu (IMJ), Centre National
  de la Recherche Scientifique (CNRS), France}
\email{freixas@math.jussieu.fr}

\author{R\u azvan Li\c tcanu}
\address{Faculty of Mathematics, University Al. I. Cuza Iasi, Romania}
\email{litcanu@uaic.ro}

\date{}

\begin{abstract}
  The classical arithmetic Grothendieck-Riemann-Roch theorem can be
  applied only 
  to projective morphisms that are smooth over the complex numbers. 
  In this paper we generalize the arithmetic Grothendieck-Riemann-Roch
  theorem to the case of general projective morphisms between regular
  arithmetic varieties. To this end we rely on the theory of
  generalized analytic torsion developed by the authors.
\end{abstract}

\maketitle

\section{Introduction}
The Grothendieck-Riemann-Roch theorem is a fundamental statement in
algebraic geometry. It describes the behavior of the Chern character
from algebraic $K$-theory to suitable cohomology theories (for
instance Chow groups), with respect to the push-forward operation by
proper maps. It provides a vast generalization of the classical
Riemann-Roch theorem on Riemann surfaces and the
Hirzebruch-Riemann-Roch theorem on compact complex manifolds. In their
development of arithmetic intersection theory, Gillet and Soul\'e were
lead to extend the Grothendieck-Riemann-Roch theorem to the context of
arithmetic varieties. In this setting, vector bundles are equipped
with additional smooth hermitian metrics, for which an extension of
algebraic $K$-theory can be defined. There is a theory of
characteristic classes for hermitian vector bundles, with values in
the so-called arithmetic Chow groups \cite{GilletSoule:vbhm}. In
analogy to the classical algebraic geometric setting, it is natural to
ask about the behavior of the arithmetic Chern character with respect
to proper push-forward. This is the question first addressed by
Gillet-Soul\'e  \cite{GilletSoule:aRRt} and later by
Gillet-R\"ossler-Soul\'e
\cite{GilletRoesslerSoule:_arith_rieman_roch_theor_in_higher_degrees}. These
works had to restrict to push-forward by morphisms which are smooth
over the complex points of the arithmetic varieties. This assumption
was necessary in defining both push-forward on arithmetic $K$-theory
and the arithmetic Chow groups. While on arithmetic Chow groups
push-forward resides on elementary operations (direct images of cycles
and fiber integrals of differential forms), they used holomorphic
analytic torsion on the arithmetic $K$-theory level
\cite{BismutKohler}.

The aim of this article is to extend the work of Gillet-Soul\'e and
Gillet-R\"ossler-Soul\'e to arbitrary projective morphisms of regular
arithmetic varieties. Hence we face to the difficulty of
non-smoothness of morphisms at the level of complex points. To
accomplish our program, we have to introduce generalized arithmetic
Chow groups and arithmetic $K$-theory groups which afford proper
push-forward functorialities for possibly non-smooth projective
morphisms. Loosely speaking, this is achieved by replacing smooth
differential forms in the theory of Gillet and Soul\'e by currents
with possibly non-empty wave front sets. To motivate the introduction
of these currents, we remark that they naturally appear as
push-forwards of smooth differential forms by morphisms whose critical
set is non-empty. In this concrete example, the wave front set of the
currents is controlled by the normal directions of the morphism. The
definition of our generalized arithmetic Chow groups is a variant of
the constructions of Burgos-Kramer-K\"uhn
\cite{BurgosKramerKuehn:cacg}, specially their covariant arithmetic
Chow groups. As an advantage with respect to \emph{loc. cit.}, the
presentation we give simplifies the definition of proper push-forward,
which is the main operation we have to deal with in the present
article. At the level of Chow groups, this operation relies on
push-forward of currents and keeps track of the wave front sets. For
arithmetic $K$-groups, we replace the analytic torsion forms of
Bismut-K\"ohler by a choice of a generalized analytic torsion theory
as developed in our previous work
\cite{BurgosFreixasLitcanu:GenAnTor}. While a generalized analytic
torsion theory is not unique, we proved it is uniquely determined by
the choice of a real genus. We establish an arithmetic
Grothendieck-Riemann-Roch theorem for arbitrary projective morphisms,
where this real genus replaces the $R$-genus of Gillet and Soul\'e. We
therefore obtain the most general possible formulation of the
theorem. In particular, the natural choice of the 0 genus,
corresponding to what we called the homogenous theory of analytic
torsion, provides an exact Grothendieck-Riemann-Roch type formula,
which is the formal translation of the classical algebraic geometric
theorem to the setting of Arakelov geometry. The present work is thus
the abutment of the articles \cite{BurgosLitcanu:SingularBC} (by the
first and third named authors) and
\cite{BurgosFreixasLitcanu:HerStruc}--\cite{BurgosFreixasLitcanu:GenAnTor}.

Let us briefly review the contents of this article. In section
\ref{section:GenArChow} we develop our new generalization of
arithmetic Chow groups, and consider as particular instances the
arithmetic Chow groups with currents of fixed wave front set. We study
the main operations, such as pull-back, push-forward and products. In
section \ref{section:ArKTheory} we carry a similar program to
arithmetic $K$-theory. We also consider an arithmetic version of our
hermitian derived categories \cite{BurgosFreixasLitcanu:HerStruc},
that is specially useful to deal with complexes of coherent sheaves
with hermitian structures. The essentials on arithmetic characteristic
classes are treated in section \ref{section:ArChar}. With the help of
our theory of generalized analytic torsion, section
\ref{section:DirectImage} builds push-forward maps on the level of
arithmetic derived categories and arithmetic $K$-theory. The last
section, namely section \ref{section:ARR}, is devoted to the statement
and proof of the arithmetic Grothendieck-Riemann-Roch theorem for
arbitrary projective morphisms of regular arithmetic varieties. As an
application, we compute the main characteristic numbers of the
homogenous theory, a question that was left open in
\cite{BurgosFreixasLitcanu:GenAnTor}.

\section{Generalized arithmetic Chow groups}\label{section:GenArChow}

Let $(A,\Sigma ,F_{\infty})$ be an arithmetic ring
\cite{GilletSoule:ait}: that is, $A$ is an excellent regular
Noetherian integral domain,
together with a finite non-empty set of embeddings $\Sigma$ of $A$
into $\CC$ and a linear 
conjugate involution $F_{\infty}$ of the product $\CC^{\Sigma}$ which
commutes with 
the diagonal embedding of $A$. 
Let $F$ be the field of fractions of $A$.
An \textit{arithmetic variety} $\mathcal{X}$ is a
flat and quasi-projective scheme over $A$ such that
$\mathcal{X}_{F}=\mathcal{X}\times \Spec F$
is smooth. Then $X_{\CC}:=\coprod_{\sigma \in \Sigma
}\mathcal{X}_{\sigma }(\CC) $ is a
complex algebraic manifold, which is endowed with an
anti-holomorphic automorphism $F_{\infty}$. One also associates to
$\mathcal{X}$ the real variety
$X=(X_{\CC},F_{\infty})$. Whenever we have arithmetic varieties
$\mathcal{X},\ \mathcal{Y}, \dots$ we will denote by $X_{\CC},\
Y_{\CC}, \dots$ the associated complex manifolds and by $X,\
Y, \dots$ the associated real manifolds.

To every regular arithmetic variety Gillet
and Soul\'e have associated arithmetic Chow groups, denoted $\cha
^{\ast}(\mathcal{X})$, and developed an
arithmetic intersection theory \cite{GilletSoule:ait}. 

The arithmetic Chow groups defined by Gillet and Soul\'e are only
covariant for morphism that are smooth on the generic fiber. Moreover
they are not suitable to study the kind of singular metrics that
appear naturally when dealing with non proper modular varieties.  
In order to have arithmetic Chow groups that are covariant with
respect to arbitrary proper morphism, or that are suitable to treat
certain kind of singular metrics,
in \cite{BurgosKramerKuehn:cacg}
different kinds
of arithmetic Chow groups are constructed, depending on the choice of a
Gillet sheaf of 
algebras $\mathcal{G}$ and a $\mathcal{G}$-complex $\mathcal{C}$. We
denote by $\cha^{\ast}(\mathcal{X},\mathcal{C})$ the arithmetic Chow groups
defined in \emph{op. cit.} Section 4.

The basic example of a Gillet algebra is the Deligne complex of
sheaves of differential forms with logarithmic singularities  
$\mathcal{D}_{\log}$, defined
in \cite[Definition 5.67]{BurgosKramerKuehn:cacg}; we refer to \emph{op. cit.}
for the precise definition and properties. Therefore, to any
$\mathcal{D}_{\log}$-complex we can associate arithmetic Chow
groups. In particular, considering $\mathcal{D}_{\log}$ itself as a
$\mathcal{D}_{\log}$-complex, we obtain
$\cha^{\ast}(\mathcal{X},\mathcal{D}_{\log})$, 
the arithmetic Chow groups defined in \cite[Section
6.1]{BurgosKramerKuehn:cacg}. When $\mathcal{X}_{F}$ is projective,
these groups agree, up to a normalization factor, with the groups
defined by Gillet and Soul\'e. 

The groups $\cha^{\ast}(\mathcal{X},\mathcal{C})$ introduced in \cite[Section
6.1]{BurgosKramerKuehn:cacg} have several technical issues:
they depend on the sheaf structure of $\mathcal{C}$ and not only on
the complex of global sections $\mathcal{C}(X)$; moreover, they are
not completely satisfactory if the cohomology determined by
$\mathcal{C}$ does not satisfy a weak purity property; finally the
definition of direct images is intricate. To overcome these
difficulties we introduce here a variant of the cohomological
arithmetic Chow groups that only depends on the complex of global
sections of a $\mathcal{D}_{\log}$-complex. 

\begin{definition}\label{def:3} Let $\mathcal{X}$ be an arithmetic variety,
  $X_{\CC}=\mathcal{X}_{\Sigma }$ the associated complex manifold and
  $X=(X_{\CC},F_{\infty})$ the associated real manifold.
  A \emph{$\mathcal{D}_{\log}(X)$-complex} is a graded complex of
  real vector spaces $C^{\ast}(\ast)$ provided with a morphism of
  graded complexes
  \begin{displaymath}
    \cmap\colon \mathcal{D}_{\log}^{\ast}(X,\ast)\longrightarrow
    C^{\ast}(\ast).  
  \end{displaymath}
 Given two $\mathcal{D}_{\log}(X)$-complexes $C$ and $C'$, we say that $C'$ is a $C$-complex if there is a commutative diagram of morphisms of graded complexes
 \begin{displaymath}
 	\xymatrix{
		\mathcal{D}_{\log}^{\ast}(X,\ast)\ar[r]^{c}\ar[rd]^{c'}	&C^{\ast}(\ast)\ar[d]^{\varphi}\\
			&C^{\prime\ast}(\ast).
	}
 \end{displaymath}
 In this situation, we say that $\varphi$ is a morphism of $\mathcal{D}_{\log}(X)$-complexes.
\end{definition}

We stress the fact that a $\mathcal{D}_{\log}$-complex is a complex of
sheaves while a $\mathcal{D}_{\log}(X)$-complex is a complex of vector
spaces. 
If $\mathcal{C}$ is a $\mathcal{D}_{\log}$-complex
  of real vector spaces, then the complex of global sections
  $\mathcal{C}^{\ast}(X,\ast)$ is  a $\mathcal{D}_{\log}(X)$-complex.
We are mainly interested in the
  $\mathcal{D}_{\log}(X)$-complexes of Example \ref{exm:1} made out of
  differential forms and 
  currents. We will follow the conventions of \cite[Section
5.4]{BurgosKramerKuehn:cacg} regarding differential forms and
currents. In particular, both the current associated to a differential form
and the current associated to a cycle have implicit a power of the trivial
period $2\pi i$.

\begin{example} \label{exm:1}
\begin{enumerate}
\item \label{item:4} The Deligne complex $\mathcal{D}^{\ast}_{\as}(X,\ast)$ of
differential forms on $X$ with arbitrary
singularities at infinity. Namely, if $E^{\ast}(X_{\CC})$ is the Dolbeault
complex 
(\cite[Definition 5.7]{BurgosKramerKuehn:cacg}) of differential forms on
$X_{\CC}$ then  
\begin{displaymath}
  \mathcal{D}_{\as}^{\ast}(X,\ast)=
\mathcal{D}^{\ast}(E^{\ast}(X_{\CC}),\ast)^{\sigma},
\end{displaymath}
where $\mathcal{D}^{\ast}(\underline{\ },\ast)$ denotes the Deligne
complex
(\cite[Definition 5.10]{BurgosKramerKuehn:cacg})
associated to a Dolbeault complex and $\sigma $
is the involution $\sigma (\eta)=\overline
{F_{\infty}^{\ast}\eta}$ as in \cite[Notation
5.65]{BurgosKramerKuehn:cacg}.

Note that $\mathcal{D}^{\ast}_{\as}(X,\ast)$ is the complex of
global sections of the $\mathcal{D}_{\log}$-complex
$\mathcal{D}^{\ast}_{\lllas}$ that appears in \cite[Section
3.6]{BurgosKramerKuehn:accavb} with empty log-log singular locus. In
particular, by \cite[Theorem 
3.9]{BurgosKramerKuehn:accavb} it satisfies the weak purity
condition. 
\item \label{item:5} The Deligne complex $\mathcal{D}^{\ast}_{\D,\as}(X,\ast)$
  of currents on $X$. Namely, if $D^{\ast}(X_{\CC})$ is the Dolbeault
  complex of currents on $X_{\CC}$ then
\begin{displaymath}
  \mathcal{D}_{\D,\as}^{\ast}(X,\ast)=
\mathcal{D}^{\ast}(D^{\ast}(X_{\CC}),\ast)^{\sigma}.
\end{displaymath}
Note that here we are considering arbitrary
currents on $X_{\CC}$ and not extendable currents as in 
\cite[Definition 6.30]{BurgosKramerKuehn:cacg}. 
\item \label{item:6} Let $T^{\ast}X_{\CC}$ be the cotangent bundle of
  $X_{\CC}$. Denote by 
  $T_{0}^{\ast}X_{\CC}=T^{\ast}X_{\CC}\setminus X_{\CC}$ 
  the cotangent bundle with the zero section deleted and let $S\subset
  T_{0}^{\ast}X_{\CC}$ be a closed conical subset that is invariant under
  $F_{\infty}$. Let $D^{\ast}(X_{\CC},S)$ be the
  complex of currents whose wave front set is contained in $S$
  \cite[Section 4]{BurgosLitcanu:SingularBC}. The  Deligne complex
  of currents on $X$ having the wave front set
  included in the fixed set $S$ is given by
\begin{displaymath}
  \mathcal{D}_{\D,\as}^{\ast}(X,S,\ast)=
\mathcal{D}^{\ast}(D^{\ast}(X_{\CC},S),\ast)^{\sigma}.
\end{displaymath}
\end{enumerate}
The maps of complexes $\mathcal{D}^{\ast}_{\as}(X,\ast)\to
\mathcal{D}^{\ast}_{\D,\as}(X,\ast)$ given by $\eta\mapsto
[\eta]$ is injective and makes of $\mathcal{D}_{\D,\as}(X,\ast)$ a $\mathcal{D}^{\ast}_{\as}(X,\ast)$-complex. We will use this map to identify
$\mathcal{D}^{\ast}_{\as}$ with a subcomplex of
$\mathcal{D}^{\ast}_{\D,\as}$. Since $\mathcal{D}^{\ast}_{\as}(X,\ast)=
\mathcal{D}^{\ast}_{\D,\as}(X,\emptyset,\ast)$ and
$\mathcal{D}^{\ast}_{D,\as}(X,\ast)= 
\mathcal{D}^{\ast}_{\D,\as}(X,T^{\ast}_{0}X,\ast)$, examples
\ref{item:4} and \ref{item:5} are particular cases of \ref{item:6}.     
\end{example}
\begin{remark}
With these examples at hand, we can specialize the definition of
$C$-complex (Definition \ref{def:3}) to $C=\mathcal{D}_{\as}(X)$. When
dealing with hermitian structures on sheaves on non-necessarily proper
varieties, we will to work with
$\mathcal{D}_{\as}(X)$-complexes rather than
$\mathcal{D}_{\log}(X)$-complexes. 
\end{remark}
We will also follow the notation of \cite[Section
3]{BurgosKramerKuehn:cacg} regarding complexes. In particular we will
write
\begin{displaymath}
  \widetilde C^{2p-1}(p)=C^{2p-1}(p)/\Img \dd_{C}, \quad
  \Zr C^{2p}(p)=\Ker \dd_{C}\cap  C^{2p}(p).
\end{displaymath}

\begin{definition}\label{def:1}
  The \emph{arithmetic Chow groups with $C$ coefficients} are defined
  as 
  \begin{equation}\label{pfformula}
    \cha^p(\mathcal{X},C)=\cha^p(\mathcal{X},\mathcal{D}_{\log})\times
    \widetilde C^{2p-1}(p)/\sim
  \end{equation}
  where $\sim$ is the equivalence relation generated by
  \begin{equation}\label{equivdef}
    (a(g),0)\sim (0,c(g)).
  \end{equation}
\end{definition}
If $\varphi:C\rightarrow C'$ is a morphism of
$\mathcal{D}_{\log}(X)$-complexes (so that $C'$ is a $C$-complex),
then there is a natural surjective morphism
\begin{equation}
  \cha^{p}(\mathcal{X},C)\times\widetilde{C^{\prime}}^{2p-1}(p)
  \longrightarrow\cha^{p}(\mathcal{X},C'). \label{eq:10}  
\end{equation}
Introducing the equivalence relation generated by
\begin{displaymath}
  ((0,c),0)\equiv ((0,0),\varphi(c)),
\end{displaymath}
we see that \eqref{eq:10} induces an isomorphism
\begin{equation}\label{eq:11}
  \cha^{p}(\mathcal{X},C)\times\widetilde{C}^{\prime 2p-1}(p)/
  \equiv\;\overset{\sim}{\longrightarrow}\cha^{p}(\mathcal{X},C').
\end{equation}

We next unwrap Definition \ref{def:1} in order to get simpler
descriptions of the arithmetic Chow groups associated to the complexes of
Example \ref{exm:1}. We start by recalling the construction of
$\cha^p(\mathcal{X},\mathcal{D}_{\log})$. The group of codimension $p$
arithmetic cycles is given by
\begin{displaymath}
  \za^{p}(\mathcal{X},\mathcal{D}_{\log})=
\left\{ (Z,\widetilde g)\in \Zr^{p}(\mathcal{X})\times 
  \widetilde{\mathcal{D}}_{\log}^{2p-1}(X\setminus \mathcal{Z}^{p},p)\left\vert
  \begin{aligned}[c]
    \dd_{\mathcal{D}}\widetilde g&\in \mathcal{D}_{\log}^{2p}(X,p)\\
    \cl(Z)&=[(\dd_{\mathcal{D}}\widetilde g,\widetilde g)]
  \end{aligned}
  \right.\right\},
\end{displaymath}
where $\Zr^{p}(\mathcal{X})$ is the group of codimension $p$
algebraic cycles of $\mathcal{X}$, $\mathcal{Z}^{p}$ is the ordered
system of codimension at least $p$ closed subsets of $X$,
\begin{displaymath}
\widetilde{\mathcal{D}}_{\log}^{2p-1}(X\setminus
\mathcal{Z}^{p},p) =\lim_{\substack{\longrightarrow\\W\in \mathcal{Z}^{p}}}
\widetilde{\mathcal{D}}_{\log}^{2p-1}(X\setminus W,p),  
\end{displaymath}
and $\cl(Z)$ and
$[(\dd_{\mathcal{D}}\widetilde g,\widetilde g)]$ denote the class in
the real Deligne-Beilinson cohomology group
$H^{2p}_{\mathcal{D},\mathcal{Z}^{p}}(X,\RR(p))$ with supports on
$\mathcal{Z}^{p}$ of the cycle $Z$ and the pair
$(\dd_{\mathcal{D}}\widetilde g,\widetilde g)$ respectively.

For each codimension $p-1$ irreducible variety $W$ and each rational
function $f\in K(W)$, there is a class $[f]\in
H^{2p-1}_{\mathcal{D}}(X\setminus |\dv f|,\RR(p))$. Hence a
class $\bmap([f])\in
\widetilde{\mathcal{D}}_{\log}^{2p-1}(X\setminus
\mathcal{Z}^{p},p)$ that is denoted $\mathfrak{g}(f)$. Then
$\rata^{p}(\mathcal{X},\mathcal{D}_{\log})$ is the group generated by the elements of
the form $\diva(f)=(\dv(f),\mathfrak{g}(f))$.

Then
\begin{displaymath}
  \cha^p(\mathcal{X},\mathcal{D}_{\log})=
  \za^{p}(\mathcal{X},\mathcal{D}_{\log})\left/
\rata^{p}(\mathcal{X},\mathcal{D}_{\log})\right. .
\end{displaymath}

We will use the following well stablished notation.
If $B$ is any subring of $\RR$ we will denote by
$\Zr^{p}_{B}(\mathcal{X})=\Zr^{p}(\mathcal{X})\otimes B$, by
$\za^{p}_{B}(\mathcal{X},\mathcal{D}_{\log})$ the group with the same
definition as $\za^{p}(\mathcal{X},\mathcal{D}_{\log})$ with
$\Zr^{p}_{B}(\mathcal{X})$ instead of $\Zr^{p}(\mathcal{X})$, and we
write
$\rata^{p}_{B}(\mathcal{X},\mathcal{D}_{\log})=
\rata^{p}(\mathcal{X},\mathcal{D}_{\log})\otimes B$. Finaly we write
\begin{equation*}
  \cha^p_{B}(\mathcal{X},\mathcal{D}_{\log})=
  \za^{p}_{B}(\mathcal{X},\mathcal{D}_{\log})\left/
\rata^{p}_{B}(\mathcal{X},\mathcal{D}_{\log})\right.
\end{equation*}
Note that
$\cha^p_{\QQ}(\mathcal{X},\mathcal{D}_{\log})=\cha^p(\mathcal{X},\mathcal{D}_{\log})\otimes
\QQ$ but, in general, $\cha^p_{\RR}(\mathcal{X},\mathcal{D}_{\log})\not
= \cha^p(\mathcal{X},\mathcal{D}_{\log})\otimes \RR$. We will use the
same notation for all variants of the arithmetic Chow groups.

Now, in the definition of $\cha^{p}(\mathcal{X},C)$ we can first
change coefficients and then take rational equivalence. We define
\begin{displaymath}
  \za^p(\mathcal{X},C)=\za^p(\mathcal{X},\mathcal{D}_{\log})\times
  \widetilde C^{2p-1}(p)/\sim
\end{displaymath}
where again $\sim$ is the equivalence relation generated by
$(a(g),0)\sim (0,c(g))$.

There are maps
\begin{alignat*}{2}
\zeta_{C}&\colon
\za^{p}(\mathcal{X},C)\longrightarrow \Zr^{p}(\mathcal{X}),
& \zeta_{C}((Z,\widetilde g),\widetilde c)&=Z,
\\
\amap_{C}&\colon \widetilde{C}^{2p-1}(p)\longrightarrow\za^{p}
(\mathcal{X},C),
& \amap_{C}(\widetilde c)&=((0,0),-\widetilde c),\\
\omega_{C}&\colon \za^{p}(\mathcal{X},C)\longrightarrow \Zr C^{2p}(p), &
\qquad \omega_{C}((Z,\widetilde g),\widetilde c)&=
\cmap(\dd_{\mathcal{D}}\widetilde g)+\dd_{C}\widetilde c.
\end{alignat*}
We also consider the map 
$$\bmap_{C}\colon H^{2p-1}(C^{\ast}(p))\to \za^{p} (\mathcal{X},C)$$
obtained by composing $\amap_{C}$ 
with the inclusion $H^{2p-1}(C^{\ast}(p))\to \widetilde C^{2p-1}(p)$,
and the map $$\rho _{C}\colon \CH^{p,p-1}(\mathcal{X})\to
H^{2p-1}(C^{\ast}(p))$$ obtained by composing the regulator map $\rho
\colon \CH^{p,p-1}(\mathcal{X})\to H^{2p-1}_{\mathcal{D}}(X,\RR(p))$
in \cite[Notation 4.12]{BurgosKramerKuehn:cacg}
with the map $\cmap\colon H^{2p-1}_{\mathcal{D}}(X,\RR(p))\to
H^{2p-1}(C^{\ast}(p))$. We will also denote by $\rho _{C}$ the
analogous map with target $\widetilde C^{2p-1}(p)$.

There are induced maps
\begin{align*}
\zeta_{C}&\colon
\cha^{p}(\mathcal{X},C)\longrightarrow\CH^{p}(\mathcal{X}),
\\
\amap_{C}&\colon \widetilde{C}^{2p-1}(p)\longrightarrow\cha^{p}
(\mathcal{X},C),\\
\omega_{C}&\colon \cha^{p}(\mathcal{X},C)\longrightarrow\Zr C^{2p}(p).
\end{align*}

\begin{lemma}\label{lemm:1}
  \begin{enumerate}
  \item \label{item:1} Let $\rata ^{p}(\mathcal{X},C)$ denote the
    image of $\rata 
    ^{p}(\mathcal{X},\mathcal{D}_{\log})$ in the group 
    $\za^p(\mathcal{X},C)$. Then
    \begin{displaymath}
      \cha^p(\mathcal{X},C)=
      \za^{p}(\mathcal{X},C)\left/
        \rata^{p}(\mathcal{X},C)\right. .      
    \end{displaymath}
  \item \label{item:2} There is an exact sequence
    \begin{equation}\label{eq:3}
      0\to \widetilde C^{2p-1}(p)\overset{\amap_{C}}{\longrightarrow
      }\za^{p}(\mathcal{X},C)
      \overset{\zeta_{C}}{\longrightarrow }
      \Zr^{p}(\mathcal{X}) \to 0.
    \end{equation}
  \item \label{item:3} There are exact sequences
\begin{equation}\label{eq:1}
\CH^{p,p-1}(\mathcal{X})\overset{\rho_{C}}{\longrightarrow }
\widetilde{C}^{2p-1}(p)\overset{\amap_{C}}{\longrightarrow }
\cha^p(\mathcal{X},C) 
\overset{\zeta_{C}}{\longrightarrow }  \CH^p(\mathcal{X})\to 0,
\end{equation}
and
\begin{multline}\label{eq:2}
  \CH^{p,p-1}(\mathcal{X})\overset{\rho_{C}}{\longrightarrow }
  H^{2p-1}(C^{\ast}(p))\overset{\bmap_{C}}{\longrightarrow }\cha^p(\mathcal{X},C)
  \overset{\zeta_{C}\oplus \omega _{C}}{\longrightarrow }
 \\  \CH^p(\mathcal{X})\oplus\Zr C^{2p}(p)
\longrightarrow  H^{2p}(C^{\ast}(p)) \to 0.
\end{multline}
 \end{enumerate}
  \end{lemma}
  \begin{proof}
    \ref{item:1} Follows easily from the definition.

    \ref{item:2} By \cite[Proposition 5.5]{Burgos:CDB} there is an
    exact sequence
    \begin{displaymath}
      0\to \widetilde
      {\mathcal{D}}_{\log}^{2p-1}(X,p)\overset{\amap}{\longrightarrow 
      }\za^{p}(\mathcal{X},\mathcal{D}_{\log})
      \overset{\zeta}{\longrightarrow }
      \Zr^{p}(\mathcal{X}) \to 0.      
    \end{displaymath}
    From it and the definition of
    $\za^{p}(\mathcal{X},\mathcal{D}_{\log})$ we derive the exactness of
    \eqref{eq:3}.  

    \ref{item:3} Follows from the exact sequences of \cite[Theorem
    7.3]{Burgos:CDB} and the definition of $\cha^p(\mathcal{X},C)$.
  \end{proof}

The contravariant functoriality of
$\cha^{\ast}(\mathcal{X},\mathcal{D}_{\log})$ is easily translated to
other coefficients.
Let $f\colon \mathcal{X}\to \mathcal{Y}$ be a morphism of regular arithmetic
varieties. 
Let $C$ be a $\mathcal{D}_{\log}(X)$-complex and $C'$ a
$\mathcal{D}_{\log}(Y)$-complex, such that
there exists a map of complexes $f^{\ast}:C'^{\ast}(\ast)\to
C^{\ast}(\ast)$ that makes the 
following diagram commutative:
\begin{displaymath}
  \xymatrix{
  \mathcal{D}_{\log}^{\ast}(Y)\ar[r]^{f^{\ast}}\ar[d] & \mathcal{D}_{\log}^{\ast}(X) \ar[d] \\
  C'^{\ast}(Y)\ar[r]_{f^{\ast}} & C^{\ast}(X).
}
\end{displaymath}
Then we define
$$f^{\ast}((\mathcal{Z},g),c)=(f^{\ast}(\mathcal{Z},g),f^{\ast}(c)).$$
It is easy to see that this map is well defined, because the pull-back map
$f^{\ast}:\cha^{\ast}(\mathcal{Y})\to \cha^{\ast}(\mathcal{X})$ (for
the Chow groups corresponding to 
$\mathcal{D}_{\log}$) is compatible to the map $\amap$. 

Before stating concrete examples of this contravariant functoriality
we need some notation (\cite[Theorem 8.2.4]{Hormander:MR1065993},
 see also \cite[Section 4]{BurgosLitcanu:SingularBC}). Let
 $f_{\CC}\colon X_{\CC}\to Y_{\CC}$ denote 
the induced map of complex manifolds. Let $N_{f}$ be the set of normal
directions of $f_{\CC}$,  that is
 \begin{displaymath}
   N_{f}=\{(f(x),\xi)\in T_{0}^{\ast}Y_{\CC}\mid \dd
   f_{\CC}^{t}\xi=0\}. 
 \end{displaymath}
 Let $S\subset
  T^{\ast}_{0}Y_{\CC}$ be a closed conical subset invariant under
  $F_{\infty}$. When
  $N_{f}\cap S=\emptyset$, the function  $f$ is said to be transverse
  to $S$. In this case we write
  \begin{displaymath}
    f^{\ast}S=\{(x,\dd f_{\CC}^{t}\xi)\mid (f(x),\xi)\in S\}.
  \end{displaymath}
  It is a closed conical subset of $T^{\ast}_{0}X_{\CC}$ invariant under
  $F_{\infty}$.

\begin{proposition} \label{prop:3}
  Let $f\colon \mathcal{X}\to \mathcal{Y}$ be a morphism of regular arithmetic
  varieties.
  \begin{enumerate}
  \item There is a pull-back morphism $f^{\ast}\colon
    \cha^{p}(\mathcal{Y},\mathcal{D}_{\as}(Y)) \to
    \cha^{p}(\mathcal{X},\mathcal{D}_{\as}(X))$.
  \item Let $N_{f}$ be the set of normal directions of $f_{\CC}$
    and $S\subset
    T^{\ast}_{0}Y_{\CC}$ a closed conical subset invariant under
    $F_{\infty}$. If $N_{f}\cap
    S=\emptyset$, then there is a pull-back morphism
    \begin{displaymath} f^{\ast}\colon
      \cha^{p}(\mathcal{Y},\mathcal{D}_{\D,\as}(Y,S)) \to
      \cha^{p}(\mathcal{X},\mathcal{D}_{\D,\as}(X,f^{\ast}S)).
    \end{displaymath}
  \item If $f_{F}$ is smooth (hence $N_{f}=\emptyset$) then there is a
    pull-back morphism 
    $$f^{\ast}\colon
    \cha^{p}(\mathcal{Y},\mathcal{D}_{\D,\as}(Y)) \to
    \cha^{p}(\mathcal{X},\mathcal{D}_{\D,\as}(X)).$$
  \end{enumerate}
\end{proposition}

Similarly the multiplicative properties of
$\cha^{\ast}_{\QQ}(\mathcal{X},\mathcal{D}_{\log})$ can be transferred to
other coefficients.  
Let $C$,  $C'$ and $C''$ be
$\mathcal{D}_{\log}(X)$-complexes such that there is a commutative
diagram of morphisms of complexes
\begin{displaymath}
  \xymatrix{
  \mathcal{D}_{\log}(X)\otimes \mathcal{D}_{\log}(X) \ar[r]^-{\bullet} \ar[d] & \mathcal{D}_{\log} (X)\ar[d] \\
  C\otimes C' \ar[r]_-{\bullet} & C''.
}
\end{displaymath}
Then we define a product
$$\cha^p(\mathcal{X},C)\times \cha^q(\mathcal{X},C')
\to \cha^{p+q}_{\QQ}(\mathcal{X},C'')$$
by
\begin{equation}\label{product}
((\mathcal{Z},g),c)\cdot((\mathcal{Z}',g'),c')=
((\mathcal{Z},g)\cdot(\mathcal{Z}',g'),c\bullet\cmap(\omega(g'))+
\cmap(\omega(g))\bullet c'+\dd_{C}c\bullet c').
\end{equation}
As a consequence we obtain the following result.
\begin{proposition} \label{prop:2}
  Let $\mathcal{X}$ be a regular arithmetic variety.
  \begin{enumerate}
  \item $\cha^{\ast}_{\QQ}(\mathcal{X},\mathcal{D}_{\as}(X))$ is an
    associative commutative graded ring.
  \item $\cha^{\ast}_{\QQ}(\mathcal{X},\mathcal{D}_{\D,\as}(X))$ is a
    module over
    $\cha^{\ast}_{\QQ}(\mathcal{X},\mathcal{D}_{\as}(X))$.
  \item Let $S,S'$ be closed conic subsets of $T^{\ast}_{0}X_{\CC}$ that are
    invariant under $F_{\infty}$. Then
    $\cha^{\ast}_{\QQ}(\mathcal{X},\mathcal{D}_{\D,\as}(X,S))$ is a
    module over $\cha^{\ast}_{\QQ}(\mathcal{X},\mathcal{D}_{\as}(X))$.
    Moreover, if $S\cap (-S')=\emptyset$, there is a graded bilinear map
    \begin{multline*}
      \cha^p(\mathcal{X},\mathcal{D}_{\D,\as}(X,S))\times
      \cha^q(\mathcal{X},\mathcal{D}_{\D,\as}(X,S')) \longrightarrow  \\
      \cha^{p+q}_{\QQ}(\mathcal{X},\mathcal{D}_{\D,\as}(X,S\cup S'\cup (S+S'))).
    \end{multline*}
  \item The product is compatible with the pull-back of Proposition
    \ref{prop:3}. 
  \end{enumerate}
\end{proposition}

We now turn our attention towards direct images. The definition of
direct images for general $\mathcal{D}_{\log}$-complexes is quite
intricate, involving the notion of covariant $f$-pseudo-morphisms (see
\cite[Definition 3.71]{BurgosKramerKuehn:cacg}). By contrast, we will
give a another description of the groups
$\cha^{\ast}_{\QQ}(\mathcal{X},\mathcal{D}_{\D,\as}(X))$ for which the
definition of push-forward is much simpler.

By \cite[3.8.2]{Burgos:Gftp} we know that any
$\mathcal{D}_{\log}$-Green form is locally integrable. Therefore there
is a well defined map
\begin{displaymath}
  \varphi\colon \za^{p}(\mathcal{X},\mathcal{D}_{\log})\longrightarrow 
  \Zr^{p}(\mathcal{X})\oplus \widetilde {\mathcal{D}}_{\D,\as}^{2p-1}(X,p)
\end{displaymath}
given by $(Z,\widetilde g)\to (Z,\widetilde{[g]})$ for any
representative $g$ of $\widetilde g$.  The previous map can be extended
to a map
\begin{displaymath}
  \varphi_{\mathcal{D}_{\D,\as}(X)}\colon
  \za^{p}(\mathcal{X},\mathcal{D}_{\D,\as}(X))\longrightarrow  
  \Zr^{p}(\mathcal{X})\oplus \widetilde  {\mathcal{D}}_{\D,\as}^{2p-1}(X,p)
\end{displaymath}
given by $\varphi_{\mathcal{D}_{\D,\as}(X)}((Z,\widetilde g),\widetilde
h)=(Z,\widetilde{[g]}+\widetilde h)$.

The following result is clear.
\begin{lemma}\label{lemm:3} The
  map $\varphi_{\mathcal{D}_{\D,\as}(X)}$ is an isomorphism. 
\end{lemma}
  This lemma gives us a more concrete description of the
  group $\za^{p}(\mathcal{X},\mathcal{D}_{\D,\as}(X))$. In fact,
  we will identify it with the group $\Zr^{p}(\mathcal{X})\oplus
  \widetilde {\mathcal{D}}_{\D,\as}^{2p-1}(X,p)$ when necessary. Some
  care has to be taken when
  doing this identification. For instance
  \begin{equation}
    \label{eq:5}
    \omega
    _{\mathcal{D}_{\D,\as}(X)}(\varphi_{\mathcal{D}_{\D,\as}(X)}^{-1}(Z,\widetilde
    g))= \dd_{\mathcal{D}}g+\delta _{Z}.
  \end{equation}
 
Let now $f\colon \mathcal{X}\to \mathcal{Y}$ be a proper morphism of
regular arithmetic varieties of relative dimension $e$. Using the
above identification, we define
\begin{equation}\label{eq:4}
  f_{\ast}\colon
  \za^{p}(\mathcal{X},\mathcal{D}_{\D,\as}(X))\longrightarrow 
  \za^{p-e}(\mathcal{Y},\mathcal{D}_{\D,\as}(Y))
\end{equation}
by $f_{\ast}(Z,\widetilde g)=(f_{\ast}Z, \widetilde{f_{\ast} g})$,
where $g$ is any representative of $\widetilde g$, and  $f_{\ast}(g)$ is
the usual direct image of currents given by
$f_{\ast}(g)(\eta)=g(f^{\ast}\eta)$.  

\begin{proposition} \label{prop:4}
  The map $f_{\ast}$ in \eqref{eq:4} sends the group $
  \rata^{p}(\mathcal{X},\mathcal{D}_{\D,\as}(X))$ to the group
  $\rata^{p-e}(\mathcal{Y},\mathcal{D}_{\D,\as}(Y))$. Therefore
  it induces a map
  \begin{displaymath}
    \cha^{p}(\mathcal{X},\mathcal{D}_{\D,\as}(X))\longrightarrow 
  \cha^{p-e}(\mathcal{Y},\mathcal{D}_{\D,\as}(Y)).
  \end{displaymath}
\end{proposition}

In order to transfer this push-forward to other coefficients we
introduce two extra properties for a $\mathcal{D}_{\log}(X)$-complex $C$.
\begin{description}
\item{(H1)} There is a commutative diagram of injective morphisms of
  complexes 
\begin{displaymath}
  \xymatrix{
    \mathcal{D}_{\log} ^{\ast}(X,\ast)\ar[r]^{\cmap}\ar[dr]&
    C^{\ast}(\ast)\ar[d]^{\cmap'}   \\
  & \mathcal{D}_{\D,\as} ^{\ast}(X,\ast).
}
\end{displaymath}
Since $\cmap'$ is injective we will usually identify $C$ with its image by
$\cmap'$.
\item{(H2)} The map $\cmap'$ induces isomorphisms
  \begin{align*}
    H^{n}(C^{\ast}(p))\cong H^{n}(\mathcal{D}_{\D,\as}^{\ast}(X,p))
  \end{align*}
for all $p\ge 0$ and $n=2p-1,2p$.
\end{description}

The conditions (H1) and (H2) have two consequences. First
if $\eta \in D_{\D,\as}^{2p-1}(X,p)$ is a current such that
$\dd_{\mathcal{D}}\eta\in C^{2p}(p)$, there exist
$a\in\mathcal{D}_{\D,\as}^{2p-2}(X,p)$ such that 
$$\dd_{\mathcal{D}}a+\eta\in C^{2p-1}(p).$$
Second, the induced map $\widetilde C^{2p-1}(p)\to \widetilde
{\mathcal{D}}^{2p-1}(X,p)$ is injective.  

Let $C$ be a  $\mathcal{D}_{\log}(X)$-complex satisfying (H1).
Consider the diagram
\begin{displaymath}
  \xymatrix{
  \cha^p(\mathcal{X},C) \ar[ddr]_{i}\ar[dr]_{j}\ar^{\omega _{C}}[drr] &&\\
  & A \ar[r] \ar[d] & \Zr C^{2p}(p)\ar[d] \\
  & \cha^p(\mathcal{X},\mathcal{D}_{\D,\as}(X)) \ar[r]_-{\omega} & \Zr
  \mathcal{D}_{\D,\as}^{2p}(p) 
}
\end{displaymath}
where $A$ is defined by the cartesian square, $i$ is induced by (H1)
and $j$ is induced by $i$ and $\omega _{C}$.
\begin{lemma} \label{lemm:2}
  If $C$ also satisfies (H2) then $j$ is an isomorphism.
\end{lemma}
\begin{proof}
By the injectivity of $\widetilde C^{2p-1}(p)\to \widetilde
{\mathcal{D}}^{2p-1}(X,p)$ and Lemma \ref{lemm:1}~\ref{item:3}, the map $i$ is
injective. Hence  the map $j$ is injective and we only need to prove
that $j$ is surjective.

Let $x=((Z,\widetilde g),\eta)\in A$. This means that $(Z,\widetilde
g)\in\cha^p(\mathcal{X},\mathcal{D}_{\D,\as}(X))$, $\eta\in \Zr
C^{2p}(p)$ and $\omega _{\mathcal{D}_{\D,\as}(X)}(Z,\widetilde
g)=\dd_{\mathcal{D}}g+\delta _{Z}=\eta$. Let $g'\in 
\mathcal{D}^{2p-1}_{\log}(X\setminus |Z|,p)$ be a Green form for
$Z$. Then $g-[g']\in \mathcal{D}^{2p-1}_{\D,\as}(X,p)$ satisfies 
$\dd_{\mathcal{D}}(g-[g'])\in C^{2p}(p)$. By (H2) there is $a\in
\mathcal{D}^{2p-2}_{\D,\as}(X,p)$ such that
$g'':=g-[g']+\dd_{\mathcal{D}}a\in C^{2p-1}(p)$.
Consider the element $x'=((Z,\widetilde g'),\widetilde g'')\in
\cha^{p}(\mathcal{X},C)$. To see that $j(x')=x$ we have to check that 
$\omega _{C}(x')=\eta$ and $i(x')=(Z,\widetilde g)$. We compute
\begin{gather*}
  \omega _{C}(x')=\cmap(\dd_{\mathcal{D}}g')+\dd_{C}g''=
  \dd_{\mathcal{D}}[g']+\delta _{Z}+\dd_{\mathcal{D}} g-\dd_{\mathcal{D}}[g']+
  \dd_{\mathcal{D}}\dd_{\mathcal{D}}a=\eta,\\
  i(x')=(Z,([g']+g-[g']+\dd_{\mathcal{D}}a)^{\sim})=(Z,\widetilde g),
\end{gather*}
concluding the proof of the lemma.
\end{proof}
We can rephrase the lemma as follows.
\begin{theorem}\label{pf=pb} 
Let $C$ be a $\mathcal{D}_{\log}(X)$-complex that satisfies (H1) and
(H2). Then the map $\varphi$ can be extended to an injective map 
\begin{displaymath}
  \varphi_{C}\colon \za^{p}(\mathcal{X},C)\longrightarrow 
  \Zr^{p}(\mathcal{X})\oplus \widetilde D_{\D,\as}^{2p-1}(X,p).
\end{displaymath}
Moreover
\begin{equation}
  \label{eq:6}
  \Img(\varphi_{C})=\left\{(Z,\widetilde g)\in
    \Zr^{p}(\mathcal{X})\oplus \widetilde
    {\mathcal{D}}_{\D,\as}^{2p-1}(X,p) \mid 
    \dd_{\mathcal{D}}g+\delta _{Z}\in
    C^{2p}(p)\right\}.   
\end{equation}
\end{theorem}
In view of this theorem, if $C$ satisfies (H1) and (H2), we will
identify $\za^{p}(\mathcal{X},C)$ with the right hand side of equation
\eqref{eq:6}.

\begin{proposition}\label{prop:1} The
  $\mathcal{D}_{\log}(X)$-complexes   $\mathcal{D}_{\as}(X)$,
  $\mathcal{D}_{\D,\as}(X,S)$ and $\mathcal{D}_{\D,\as}(X)$ satisfy
  (H1) and (H2). Therefore we can identify
\begin{align*}
  \za^{p}(\mathcal{X},\mathcal{D}_{\as}(X))&\cong
  \left\{(Z,\widetilde g) \in \Zr^{p}\oplus \widetilde
    {\mathcal{D}}_{\D,\as}^{2p-1}(X,p)
    \mid
    \dd_{\mathcal{D}}g+\delta _{Z}\in
    \mathcal{D}^{2p}_{\as}(X,p)\right\}\\
  \za^{p}(\mathcal{X},\mathcal{D}_{\D,\as}(X,S))&\cong
  \left\{(Z,\widetilde g) \in \Zr^{p}\oplus \widetilde
    {\mathcal{D}}_{\D,\as}^{2p-1}(X,p)
    \mid 
    \dd_{\mathcal{D}}g+\delta _{Z}\in
    \mathcal{D}^{2p}_{\D,\as}(X,S,p)\right\}\\
  \za^{p}(\mathcal{X},\mathcal{D}_{\D,\as}(X))&\cong
  \Zr^{p}(\mathcal{X})\oplus \widetilde
    {\mathcal{D}}_{\D,\as}^{2p-1}(X,p).
\end{align*}
\end{proposition}
\begin{proof}
  The result for $\mathcal{D}_{\as}(X)$ follows from the Poincar\'e
  $\ov \partial$-Lemma for currents
  \cite[Pag. 382]{GriffithsHarris:pag}. The result for
  $\mathcal{D}_{\D,\as}(X,S)$ follows from the Poincar\'e
  $\ov \partial$-Lemma for currents with fixed wave front set
  \cite[Theorem 4.5]{BurgosLitcanu:SingularBC}. The statement for
  $\mathcal{D}_{\D,\as}(X)$ is Lemma \ref{lemm:3}.
\end{proof}

\begin{remark}
  The previous proposition gives us a more concrete description of the
  groups $\za^{p}(\mathcal{X},C)$ for $C=\mathcal{D}_{\as}(X)$,
  $\mathcal{D}_{\D,\as}(X,S)$ and $\mathcal{D}_{\D,\as}(X)$ and
  also a more concrete description of the groups
  $\cha^{p}(\mathcal{X},C)$. In particular it easily implies that the
  groups $\cha^{p}(\mathcal{X},\mathcal{D}_{\as}(X))$, $p\ge 0$, agree
  (up to a 
  normalization factor) with the arithmetic Chow groups of Gillet-Soul\'e
  $\cha^{p}(\mathcal{X})$ and that the groups
  $\cha^{p}(\mathcal{X},\mathcal{D}_{\D,\as}(X))$, $p\ge 0$, agree
  (again up to a 
  normalization factor) with the
  arithmetic Chow groups of Kawaguchi-Moriwaki $\cha^{p}_{D}(\mathcal{X})$, introduced in
  \cite{KawaguchiMoriwaki:isfav}.   
\end{remark}

We want to transfer the push-forward of Proposition \ref{prop:4} to
complexes satisfying (H1) and (H2). 
Let again $f\colon \mathcal{X}\to \mathcal{Y}$ be a proper morphism of
regular arithmetic varieties of relative dimension $e$. Let 
$C$ be a $\mathcal{D}_{\log}(X)$-complex and $C'$ a
$\mathcal{D}_{\log}(Y)$-complex, both satisfying (H1) and (H2). We assume  
there exists a map of complexes $f_{\ast}\colon C^{\ast}(\ast)\to
C'^{\ast-2e}(\ast-e)$ that makes the 
following diagram commutative: 
\begin{equation}\label{eq:18}
  \xymatrix{
  C^{\ast}(\ast)\ar[r]^-{f_{\ast}} \ar[d] & C'^{\ast-2e}(\ast-e) \ar[d] \\
  \mathcal{D}_{\D}^{\ast}(X,\ast)\ar[r]_-{f_{\ast}} & \mathcal{D}_{\D}^{\ast-2e}(Y,\ast-e).
} 
\end{equation}
Then, for $p\ge 0$, we define maps
\begin{displaymath}
  f_{\ast}\colon \za^{p}(\mathcal{X},C)\to \za^{p-e}(\mathcal{Y},C') 
\end{displaymath}
that, using the identification of Theorem \ref{pf=pb} are given by
\begin{equation}\label{dir_im_gen}
f_{\ast}(Z,\widetilde g)=(f_{\ast}Z,\widetilde {f_{\ast}(g)}).
\end{equation}
By Proposition \ref{prop:4}, these maps send $
\rata^{p}(\mathcal{X},C)$ to 
$\rata^{p-e}(\mathcal{Y},C')$. Therefore they
induce morphisms
\begin{displaymath}
  f_{\ast}\colon \cha^{p}(\mathcal{X},C)\longrightarrow 
  \cha^{p-e}(\mathcal{Y},C').
\end{displaymath}

Before stating concrete examples of this covariant functoriality
we need some notation.  Let
 $f_{\CC}\colon X_{\CC}\to Y_{\CC}$ denote 
the induced proper map of complex manifolds.
 If $S\subset
  T^{\ast}_{0}X_{\CC}$ is a closed conical subset invariant under
  $F_{\infty}$, then we write
  \begin{displaymath}
    f_{\ast}S=\{(f(x),\xi)\in T^{\ast}_{0}Y_{\CC}\mid (x,\dd
    f^{t}\xi)\in S\}\cup N_{f}. 
  \end{displaymath}
  It is a closed conical subset of $T^{\ast}_{0}Y_{\CC}$ invariant under
  $F_{\infty}$.

  \begin{proposition}\label{prop:8}
    Let $f\colon \mathcal{X}\to \mathcal{Y}$, $g\colon \mathcal{Y}\to\mathcal{Z}$ be proper morphism of
regular arithmetic varieties of relative dimension $e$ and $e'$, and $S\subset
  T^{\ast}_{0}X_{\CC}$ a closed conical subset invariant under
  $F_{\infty}$. Then there are maps
  \begin{displaymath}
    f_{\ast}\colon \cha^{p}(\mathcal{X},\mathcal{D}_{\D,\as}(X,S))
    \to \cha^{p-e}(\mathcal{Y},\mathcal{D}_{\D,\as}(Y,f_{\ast}S)),
  \end{displaymath}
  and similarly
  \begin{displaymath}
  	g_{\ast}\colon \cha^{p'}(\mathcal{Y},\mathcal{D}_{\D,\as}(Y,f_{\ast}S))
    \to \cha^{p'-e'}(\mathcal{Z},\mathcal{D}_{\D,\as}(Z,g_{\ast}f_{\ast}S)).
  \end{displaymath}
  Furthermore, the relation $(g\circ f)_{\ast}=g_{\ast}\circ f_{\ast}$ is satisfied.
  \end{proposition}

As a particular case of the above proposition we obtain the following
cases.
\begin{corollary}\label{cor:1}
     Let $f\colon \mathcal{X}\to \mathcal{Y}$ be a proper morphism of
regular arithmetic varieties of relative dimension $e$.
\begin{enumerate}
\item  There are maps
  \begin{displaymath}
    f_{\ast}\colon \cha^{p}(\mathcal{X},\mathcal{D}_{\as}(X))
    \to \cha^{p-e}(\mathcal{Y},\mathcal{D}_{\D,\as}(Y,N_{f})).
  \end{displaymath}
\item If $N_{f}=\emptyset$, then there are maps 
  \begin{displaymath}
    f_{\ast}\colon \cha^{p}(\mathcal{X},\mathcal{D}_{\as}(X))
    \to \cha^{p-e}(\mathcal{Y},\mathcal{D}_{\as}(Y)).
  \end{displaymath}
\end{enumerate}
\end{corollary}

The functoriality and the multiplicative structures satisfy the
following compatibility properties.

\begin{theorem}\label{thm:1}
  Let $f\colon \mathcal{X}\to \mathcal{Y}$ be a morphism of regular arithmetic
  varieties. Let $N_{f}$ be the set of normal directions of $f_{\CC}$ and
  $S,\,S'\subset
  T^{\ast}_{0}Y_{\CC}$  closed conical subsets invariant under
  $F_{\infty}$. Then
  \begin{displaymath}
    f^{\ast}(S\cup
        S'\cup 
        (S+S'))=f^{\ast}(S)\cup f^{\ast}(
        S')\cup 
        (f^{\ast}(S)+f^{\ast}(S'))).
  \end{displaymath}
  If $N_{f}\cap(S\cup S'\cup S+S')=\emptyset$ and $S\cap
  (-S')=\emptyset$ then $f^{\ast}(S)\cap
  (-f^{\ast}(S'))=\emptyset$. In
  this case, if $\alpha \in
  \cha^p(\mathcal{Y},\mathcal{D}_{\D,\as}(Y,S))$ and $\beta \in
  \cha^q(\mathcal{Y},\mathcal{D}_{\D,\as}(Y,S'))$ then
  \begin{displaymath}
    f^{\ast}(\alpha \cdot \beta )=f^{\ast}(\alpha )\cdot
    f^{\ast}(\beta )\in
    \cha^{p+q}_{\QQ}(\mathcal{X},\mathcal{D}_{\D,\as}(X,f^{\ast}(S\cup
    S'\cup 
    (S+S')))). 
  \end{displaymath}
\end{theorem}

\begin{theorem}\label{thm:2}
    Let $f\colon \mathcal{X}\to \mathcal{Y}$ be a proper morphism of
    regular arithmetic 
  varieties of relative dimension $e$. Let $N_{f}$ be the set of
  normal directions of $f_{\CC}$,
  $S\subset
  T^{\ast}_{0}X_{\CC}$ and   $S'\subset
  T^{\ast}_{0}Y_{\CC}$  closed conical subsets invariant under
  $F_{\infty}$. Then
  \begin{displaymath}
    f_{\ast}(S\cup f^{\ast}(
        S')\cup 
        (S+f^{\ast}(S')))\subset f_{\ast}(S)\cup
        S'\cup 
        (f_{\ast}(S)+S')).
  \end{displaymath}
  If $f_{\ast}(S)\cap (-S')=\emptyset$ then $N_{f}\cap
  S'=\emptyset$ and $S\cap (-f^{\ast}(S'))=\emptyset$. In this case, if
  $\alpha \in 
  \cha^p(\mathcal{X},\mathcal{D}_{\D,\as}(X,S))$ and $\beta \in
      \cha^q(\mathcal{Y},\mathcal{D}_{\D,\as}(Y,S'))$ then
      \begin{displaymath}
        f_{\ast}(\alpha \cdot f^{\ast}(\beta) )=f_{\ast}(\alpha )\cdot
        \beta \in
        \cha^{p+q-e}_{\QQ}(\mathcal{Y},\mathcal{D}_{\D,\as}(Y,f_{\ast}(S)\cup
        S'\cup 
        (f_{\ast}(S)+S'))). 
      \end{displaymath} 
\end{theorem}

\section{Arithmetic $K$-theory and derived categories}\label{section:ArKTheory}
\subsection{Arithmetic $K$-theory}
As for arithmetic Chow groups, the arithmetic $K$-groups of
Gillet-Soul\'e can be generalized to include more general
coefficients at the archimedean places. Because the
definition of arithmetic $K$-groups involves hermitian vector bundles
whose metrics have arbitrary singularities at infinity, we are
actually forced to consider $\mathcal{D}_{\as}(X)$-complex
coefficients. The reader is referred to
\cite[Sec. 4.2]{BurgosKramerKuehn:accavb} for the construction of the
groups $\widehat{K}_{0}(\mathcal{X},\mathcal{D}_{\as})$, which provide
the base of our extension (see also Definition \ref{def:2} below). An
arithmetic Chern character allows to compare these generalized
arithmetic $K$-groups and the generalized arithmetic Chow groups. The
arithmetic Chern character is automatically compatible with pull-back
and products whenever defined.

\begin{definition}\label{def:2}
Let $\mathcal{X}$ be an arithmetic variety and $C^{\ast}(\ast)$ a $\mathcal{D}_{\as}(X)$-complex with structure morphism $\cmap:\mathcal{D}_{\as}^{\ast}(X,\ast)\rightarrow C^{\ast}(\ast)$. The arithmetic $K$ group of $\mathcal{X}$ with $C$ coefficients is the abelian group $\widehat{K}_{0}(\mathcal{X},C)$ generated by pairs $(\ov{\mathcal{E}},\eta)$, where $\ov{\mathcal{E}}$ is a smooth hermitian vector bundle on $\mathcal{X}$ and $\eta\in\bigoplus_{p\geq 0}\widetilde{C}^{2p-1}(p)$, modulo the relations
\begin{displaymath}
	(\ov{\mathcal{E}}_{1},\eta_{1})+(\ov{\mathcal{E}}_{2},\eta_{2})=(\ov{\mathcal{E}},\cmap(\widetilde{\ch}(\ov{\varepsilon}))+\eta_{1}+\eta_{2}),
\end{displaymath}
for every exact sequence
\begin{displaymath}
	\ov{\varepsilon}\colon\quad 0\longrightarrow\ov{\mathcal{E}}_{1}\longrightarrow\ov{\mathcal{E}}\longrightarrow\ov{\mathcal{E}}_{2}\longrightarrow 0
\end{displaymath}
with Bott-Chern secondary class $\widetilde{\ch}(\ov{\varepsilon})$.
\end{definition}
An equivalent construction can be given in the same lines as for the generalized arithmetic Chow groups. For this, observe there is a natural morphism
\begin{equation}\label{eq:7}
	\begin{split}
	\widehat{K}_{0}(\mathcal{X},\mathcal{D}_{\as})&\longrightarrow\widehat{K}_{0}(\mathcal{X},C)\\
		[\ov{\mathcal{E}},\eta]&\longmapsto [\ov{\mathcal{E}},\cmap(\eta)]
	\end{split}
\end{equation}
induced by $\cmap:\mathcal{D}_{\as}^{\ast}(X,\ast)\rightarrow C^{\ast}(\ast)$, and also
\begin{equation}\label{eq:8}
	\begin{split}
		\bigoplus_{p\geq 0}\widetilde{C}^{2p-1}(p)&\longrightarrow\widehat{K}_{0}(\mathcal{X},C)\\
		\eta&\longmapsto [0,\eta].
	\end{split}
\end{equation}
One easily sees the maps \eqref{eq:7}--\eqref{eq:8} induce a natural isomorphism of groups
\begin{equation}\label{eq:9}
	\widehat{K}_{0}(\mathcal{X},\mathcal{D}_{\as})\times\bigoplus_{p\geq 0}\widetilde{C}^{2p-1}(p)/\equiv\;
	\overset{\cong}{\longrightarrow}\widehat{K}_{0}(\mathcal{X},C),
\end{equation}
where $\equiv$ is the equivalence relation generated by
\begin{displaymath}
	((\ov{\mathcal{E}},\eta),0)\equiv ((\ov{\mathcal{E}},0),\cmap(\eta)).
\end{displaymath}
Generalized arithmetic $K$-groups for suitable complexes have
pull-backs and products. Let $f:\mathcal{X}\to\mathcal{Y}$ be a morphism
of arithmetic varieties, and suppose given $\mathcal{D}_{\as}(X)$ and
$\mathcal{D}_{\as}(Y)$ complexes $C$ and $C'$ respectively, for which
there is a commutative diagram 
\begin{equation}\label{eq:12}
  \xymatrix{
    \mathcal{D}_{\as}^{\ast}(Y)\ar[d]\ar[r]^{f^{\ast}}		
    &\mathcal{D}_{\as}^{\ast}(X)\ar[d]\\
    C^{\prime\ast}\ar[r]_{f^{\ast}}
    &C^{\ast}.
  }
\end{equation}
By description \eqref{eq:9} and the contravariant functoriality of
$\widehat{K}_{0}(\underline{\ },\mathcal{D}_{\as})$, we see there is
an induced morphism of groups 
\begin{displaymath}
	f^{\ast}:\widehat{K}_{0}(\mathcal{Y},C^{\prime})
        \longrightarrow\widehat{K}_{0}(\mathcal{X},C).
\end{displaymath}
For this kind of functoriality, an analog statement to Proposition
\ref{prop:3} holds, and we leave to the reader the task of stating
it. As for products, let $C$, $C'$ and $C''$ be
$\mathcal{D}_{\as}(X)$-complexes with a product $C\otimes
C'\overset{\bullet}{\rightarrow} C''$ and a commutative diagram
\begin{equation}\label{eq:19}
	\xymatrix{
		\mathcal{D}_{\as}(X)\otimes\mathcal{D}_{\as}(X)\ar[r]^-{\bullet}\ar[d]	&\mathcal{D}_{\as}(X)\ar[d]\\
		C\otimes C'\ar[r]_-{\bullet}	&C''.
	}
\end{equation}
Then there is an induced product at the level of $\widehat{K}_{0}$
\begin{displaymath}
	\widehat{K}_{0}(\mathcal{X},C)\times\widehat{K}_{0}(\mathcal{X},C')\longrightarrow\widehat{K}_{0}(\mathcal{X},C'')
\end{displaymath}
described by the rule
\begin{displaymath}
	[\ov{\mathcal{E}},\eta]\cdot [\ov{\mathcal{E}}',\eta']=
		[\ov{\mathcal{E}}\otimes\ov{\mathcal{E}}',c(\ch(\ov{\mathcal{E}}))\bullet\eta'+c'(\ch(\ov{\mathcal{E}}'))\bullet\eta
		+\dd_{C}\eta\bullet\eta'].
\end{displaymath}
With respect to this laws, the groups $\widehat{K}_{0}$ enjoy of the
analogue properties to Proposition \ref{prop:2}. 

Finally, we discuss on the arithmetic Chern character. If
$\mathcal{X}$ is a regular arithmetic variety, we recall there is an
isomorphism of rings \cite[Thm. 4.5]{BurgosKramerKuehn:accavb}, namely 
\begin{displaymath}
  \chh:\widehat{K}_{0}(\mathcal{X},\mathcal{D}_{\as})_{\QQ}\longrightarrow
  \bigoplus_{p\geq 0}\cha^{p}_{\QQ}(\mathcal{X},\mathcal{D}_{\as}).
\end{displaymath}
If $C$ is a $\mathcal{D}_{\as}(X)$-complex, by the presentations
\eqref{eq:11} of $\cha^{\ast}(\mathcal{X},C)$ and \eqref{eq:9} of
$\widehat{K}_{0}$ it is clear that $\chh$ extends to an  isomorphism
of groups 
\begin{displaymath}
	\chh:\widehat{K}_{0}(\mathcal{X},C)_{\QQ}\longrightarrow\bigoplus_{p\geq 0}\cha^{p}_{\QQ}(\mathcal{X},C).
\end{displaymath}
Suppose now that $C,C',C''$ are $\mathcal{D}_{\as}(X)$-complexes with
a product $C\otimes C'\rightarrow C''$ as above. Then, it is easily
seen that there is a commutative diagram of morphisms of groups
\begin{displaymath}
  \xymatrix{
    \widehat{K}_{0}(\mathcal{X},C)_{\QQ}\times\widehat{K}_{0}(\mathcal{X},C')_{\QQ}
    \ar[r]^{\hspace{1cm}\cdot}\ar[d]_{(\chh,\chh)}
    &\widehat{K}_{0}(\mathcal{X},C'')_{\QQ}\ar[d]^{\chh}\\ 
    \bigoplus_{p}\cha^{p}_{\QQ}(\mathcal{X},C)\times\bigoplus_{p}
    \cha^{p}_{\QQ}(\mathcal{X},C')\ar[r]^{\hspace{1.3cm\cdot}}
    &\bigoplus_{p}\cha^{p}_{\QQ}(\mathcal{X},C''). 
  }
\end{displaymath}
This is in particular true for complexes of currents with controlled wave front set, a result that we next record.
 \begin{proposition}
Let $S, S'$ be closed conical subsets of $T_{0}^{\ast}X_{\CC}$
invariant under the action of complex conjugation, with $S\cap
(-S')=\emptyset$. Define $T=S\cup S'\cup (S+S')$. If $\alpha \in
\widehat{K}_{0}(\mathcal{X},\mathcal{D}_{\D,\as}(X,S))_{\QQ}$ and
$\beta \in
\widehat{K}_{0}(\mathcal{X},\mathcal{D}_{\D,\as}(X,S'))_{\QQ}$, then
\begin{displaymath}
  \chh(\alpha )\cdot \chh(\beta )=\chh(\alpha \cdot \beta )\in
  \bigoplus_{p}\cha^{p}_{\QQ}(\mathcal{X},\mathcal{D}_{\D,\as}(X,T)). 
\end{displaymath}
\end{proposition}
Another feature of the Chern character is its compatibility with
pull-back functoriality. Let $f:\mathcal{X}\rightarrow\mathcal{Y}$ be
a morphism of regular arithmetic varieties. Let $C,C'$ be
$\mathcal{D}_{\as}(X)$ and $\mathcal{D}_{\as}(Y)$ complexes,
respectively, together with a morphism of complexes
$f^{\ast}:C'\rightarrow C$ satifying the commutativity
\eqref{eq:12}. Then there is a commutative diagram
\begin{displaymath}
  \xymatrix{
    \widehat{K}_{0}(\mathcal{Y},C')_{\QQ}\ar[r]^{f^{\ast}}\ar[d]_{\chh}
    &\widehat{K}_{0}(\mathcal{X},C)_{\QQ}\ar[d]^{\chh}\\ 
    \cha_{\QQ}^{\ast}(\mathcal{Y},C')\ar[r]^{f^{\ast}}
    &\cha_{\QQ}^{\ast}(\mathcal{X},C).
  }
\end{displaymath}
Again, the proof is a simple consequence for the known compatibility
in the case of the $\mathcal{D}_{\as}$ arithmetic Chow groups. In
particular we have the following proposition.
\begin{proposition}\label{prop:6}
  Let $f\colon \mathcal{X}\rightarrow\mathcal{Y}$ be a morphism of regular
  arithmetic varieties, and $S\subset T_{0}^{\ast}Y_{\CC}$ a closed
  conical subset, invariant under complex conjugation and disjoint
  with the normal directions $N_{f}$ of $f_{\CC}$. Then there is a
  commutative diagram
\begin{displaymath}
  \xymatrix{
    \widehat{K}_{0}(\mathcal{Y},\mathcal{D}_{\D,\as}(Y,S))_{\QQ}\ar[r]^{f^{\ast}}\ar[d]_{\chh}	
    &\widehat{K}_{0}(\mathcal{X},\mathcal{D}_{\D,\as}(X,f^{\ast}(S))_{\QQ}\ar[d]^{\chh}\\
    \cha_{\QQ}^{\ast}(\mathcal{Y},\mathcal{D}_{\D,\as}(Y,S))\ar[r]^{f^{\ast}}
    &\cha_{\QQ}^{\ast}(\mathcal{X},\mathcal{D}_{\D,\as}(X,f^{\ast}(S)). 
  }
\end{displaymath}
\end{proposition}
\subsection{Arithmetic derived categories}
For the problem of defining direct images on arithmetic $K$-theory it
is useful to deal with arbitrary complexes of coherent sheaves instead
of locally free sheaves. We therefore introduce an arithmetic
counterpart of our theory of hermitian structures on derived
categories of coherent sheaves, developed in
\cite{BurgosFreixasLitcanu:HerStruc}, and the $\hDb$ categories in
\cite{BurgosFreixasLitcanu:GenAnTor}. We then compare this
construction to the arithmetic $K$ groups. 

\begin{definition}
Let $S$ be a scheme. We denote by $\Db(S)$ the derived category of cohomological complexes of quasi-coherent sheaves with bounded coherent cohomology.
\end{definition}
If $\mathcal{X}$ is a regular arithmetic variety, every object of
$\Db(\mathcal{X})$ is quasi-isomorphic to a bounded cohomological
complex of locally free sheaves. One checks there is a well defined
map 
\begin{displaymath}
	\Ob\Db(\mathcal{X})\longrightarrow K_{0}(\mathcal{X})
\end{displaymath}
that sends an object $\mathcal{F}^{\ast}$ to the class
$\sum_{i}(-1)^{i}[\mathcal{E}^{i}]$, where $\mathcal{E}^{\ast}$ is
quasi-isomorphic to $\mathcal{F}^{\ast}$, and that is compatible with
derived tensor products on $\Db(\mathcal{X})$ and the ring structure
on $K_{0}(\mathcal{X})$. Our aim is thus to extend this picture and
incorporate hermitian structures. 

Let us consider the complex quasi-projective manifold $X_{\CC}$
associated to an arithmetic variety $\mathcal{X}$. It comes equipped
with the conjugate-linear involution $F_{\infty}$. Recall the data
$X=(X_{\CC},F_{\infty})$ uniquely determines a smooth quasi-projective
scheme $X_{\RR}$ over $\RR$, whose base change to $\CC$ is isomorphic
to $X_{\CC}$, and such that its natural automorphism given by complex
conjugation gets identified to $F_{\infty}$. The abelian category of
quasi-coherent (resp. coherent) sheaves over $X_{\RR}$ is equivalent
to the category of quasi-coherent (resp. coherent) sheaves over
$X_{\CC}$, equivariant with respect to the action of
$F_{\infty}$. Namely, giving a quasi-coherent (resp. coherent) sheaf
on $X_{\RR}$ is equivalent to giving a quasi-coherent (resp. coherent)
sheaf $\mathcal{F}$ on $X_{\CC}$, together with a morphism of
sheaves
\begin{displaymath}
	\mathcal{F}\longrightarrow F_{\infty\ast}\mathcal{F}
\end{displaymath}
compatible with the conjugate-linear morphism
\begin{displaymath}
	F_{\infty}^{\sharp}:\OO_{X_{\CC}}\longrightarrow F_{\infty\ast}\OO_{X_{\CC}}
\end{displaymath}
induced by the morphism of $\RR$-schemes $F_{\infty} \colon
X_{\CC}\rightarrow X_{\CC}$. A similar condition characterizes
morphisms of quasi-coherent (resp. coherent) sheaves. Therefore, we
denote the bounded derived category of coherent sheaves on $X_{\RR}$
just by $\Db(X)$.

The theory of hermitian structures on the bounded derived category of
coherent sheaves on a complex algebraic manifold developed in
\cite{BurgosFreixasLitcanu:HerStruc} can be adapted to the real
situation of $\Db(X)$, by considering hermitian structures invariant
under the action of complex conjugation. All the results in
\emph{loc. cit.} carry over to the real case. We denote by $\oDb(X)$
the category whose objects are objects of $\Db(X)$ endowed with a
hermitian structure (\emph{loc. cit.}, Def. 3.10) invariant under
complex conjugation, and whose morphisms are just morphisms in
$\Db(X)$. Thus, every object $\ov{\mathcal{F}}^{\ast}$ in $\oDb(X)$ is
represented by a quasi-isomorphism
$\ov{\mathcal{E}}^{\ast}\dashrightarrow\mathcal{F}^{\ast}$, where
$\ov{\mathcal{E}}^{\ast}$ is a bounded complex of hermitian locally
free sheaves on $X_{\CC}$, equivariant under $F_{\infty}$. There is an
obvious forgetful functor $\mathfrak{F}:\oDb(X)\rightarrow\Db(X)$,
that makes of $\oDb(X)$ a principal fibered category over $\Db(X)$,
with structural group $\ov{\KA}(X)$, the group of hermitian structures
over the 0 object \cite[Def. 2.34,
Thm. 3.13]{BurgosFreixasLitcanu:HerStruc}.

Base change to $\RR$ induces a covariant functor
$\Db(\mathcal{X})\rightarrow\Db(X)$.

\begin{definition}
We define the category $\oDb(\mathcal{X})$ as the fiber product category
\begin{displaymath}
  \xymatrix{
    \Db(\mathcal{X})\times_{\Db(X)}\oDb(X)\ar[r]\ar[d]	&\oDb(X)\ar[d]^{\mathfrak{F}}\\
    \Db(\mathcal{X})\ar[r]	&\Db(X).	
  }
\end{displaymath}
We still denote by $\mathfrak{F}$ the forgetful functor $\oDb(\mathcal{X})\rightarrow\Db(\mathcal{X})$.
\end{definition}
By construction, $\mathfrak{F}$ makes of $\oDb(\mathcal{X})$ a principal fibered category over $\Db(\mathcal{X})$, with structure group $\ov{\KA}(X)$. In particular, $\ov{\KA}(X)$ acts on $\oDb(\mathcal{X})$.

The Bott-Chern secondary character $\cht$ can be defined at the level of $\ov{\KA}$ groups \cite[Sec. 4, Def. 4.6]{BurgosFreixasLitcanu:HerStruc}. In our situation, we actually have a morphism of groups
\begin{equation}\label{eq:13}
	\cht:\ov{\KA}(X)\longrightarrow\bigoplus_{p}\widetilde{\mathcal{D}}_{\as}^{2p-1}(X,p).
\end{equation}
More generally, if $C$ is a $\mathcal{D}_{\as}(X)$-complex, we may consider a secondary Chern character with values in $\bigoplus_{p}\widetilde{C}^{2p-1}(p)$, that we denote $\cht_{C}$. In particular, $\ov{\KA}(X)$ acts on $\bigoplus_{p}\widetilde{C}^{2p-1}(p)$ through $\cht_{C}$.
\begin{definition}
The arithmetic derived category $\hDb(\mathcal{X},C)$ is defined as the cartesian product
\begin{displaymath}
	\oDb(\mathcal{X})\times_{\ov{\KA}(X),\cht_{C}}\bigoplus_{p}\widetilde{C}^{2p-1}(p).
\end{displaymath}
\end{definition}
\begin{remark}\label{rem:1}
If $\mathcal{X}$ is regular, then every object of $\hDb(\mathcal{X},C)$ can be represented by 
\begin{displaymath}
	(\mathcal{F}^{\ast},\ov{\mathcal{E}}_{\CC}\dashrightarrow\mathcal{F}^{\ast}_{\CC},\widetilde{\eta})
\end{displaymath}
where $\mathcal{E}^{\ast}\dashrightarrow\mathcal{F}^{\ast}$ is any quasi-isomorphism from a bounded complex of locally free sheaves over $\mathcal{X}$. Indeed, $\mathcal{X}$ is assumed to be regular, so that $\mathcal{F}^{\ast}$ is quasi-isomorphic to a bounded complex of locally free sheaves $\mathcal{E}^{\ast}$. One then endows the $\mathcal{E}^{i}$ with smooth hermitian metrics invariant under complex conjugation, and takes into account that $\ov{\KA}(X)$ acts transitively on the hermitian structures on $\mathcal{F}^{\ast}_{\CC}$. We introduce the simplified notation $(\ov{\mathcal{E}}^{\ast}\dashrightarrow\mathcal{F},\widetilde{\eta})$ for such representatives.
\end{remark}
The categories $\hDb(\mathcal{X},C)$ are the arithmetic analogues to those we introduced in \cite[Sec. 4]{BurgosFreixasLitcanu:GenAnTor}, and have similar properties. We are only going to review some of them.

For the next proposition, we recall that any group can be considered
as a category, with morphisms given by the group law. This in
particular the case of $\widehat{K}_{0}(\mathcal{X},C)$.
\begin{theorem}
If $\mathcal{X}$ is regular, there is a natural functor
\begin{displaymath}
  \hDb(\mathcal{X},C)\longrightarrow\widehat{K}_{0}(\mathcal{X},C),
\end{displaymath}
that, at the level of objects is given by
\begin{displaymath}
  (\ov{\mathcal{E}}^{\ast}\dashrightarrow\mathcal{F},\widetilde{\eta})\longmapsto
  \left[\sum_{i}(-1)^{i}\ov{\mathcal{E}}^{i},\widetilde{\eta}\right],
\end{displaymath}
and at the level of morphism sends any $f\in \Hom_{\hDb(\mathcal{X},C)}(A,B)$ to $[B]-[A]$.
\end{theorem}
\begin{proof}
It is enough to see that the defining assignment does not depend on
the representatives. For this, take two objects
$(\ov{\mathcal{E}}^{\ast}\dashrightarrow\mathcal{F}^{\ast},\widetilde{\eta})$
and
$(\ov{\mathcal{E}}^{\prime\ast}\dashrightarrow\mathcal{F}^{\ast},\widetilde{\eta}^{\prime})$
giving raise to the same class. We have an equivalence of objects
\begin{displaymath}
  \begin{split}
    (\ov{\mathcal{E}}^{\ast}\dashrightarrow\mathcal{F}^{\ast},\widetilde{\eta})\sim&
    (\ov{\mathcal{E}}^{\prime\ast}\dashrightarrow\mathcal{F}^{\ast},\widetilde{\eta}^{\prime})\\
    &\sim
    (\ov{\mathcal{E}}^{\ast}\dashrightarrow\mathcal{F}^{\ast},\widetilde{\eta}^{\prime}+c(\widetilde{\ch}(\Id\colon
    \ov{\mathcal{F}}^{\ast}\rightarrow\ov{\mathcal{F}}^{\prime\ast}))).
  \end{split}
\end{displaymath}
Consequently
\begin{displaymath}
	\widetilde{\eta}=\widetilde{\eta}^{\prime}+c(\widetilde{\ch}(\Id\colon\ov{\mathcal{F}}^{\ast}\rightarrow\ov{\mathcal{F}}^{\prime\ast})).
\end{displaymath}
Hence, the image of $(\ov{\mathcal{E}}^{\ast}\dashrightarrow\mathcal{F}^{\ast},\widetilde{\eta})$ is equivalently written
\begin{displaymath}
	\left[\sum_{i}(-1)^{i}\ov{\mathcal{E}}^{i},\widetilde{\eta}^{\prime}+c(\widetilde{\ch}(\Id\colon\ov{\mathcal{F}}^{\ast}\rightarrow\ov{\mathcal{F}}^{\prime\ast}))\right].
\end{displaymath}
We thus have to show the equality in $\widehat{K}_{0}(\mathcal{X},C)$
\begin{displaymath}
	[0,c(\widetilde{\ch}(\Id\colon\ov{\mathcal{F}}^{\ast}\rightarrow\ov{\mathcal{F}}^{\prime\ast}))]\overset{?}{=}\left[\sum_{i}(-1)^{i}\ov{\mathcal{E}}^{\prime
            i},0\right]-\left[\sum_{i}(-1)^{i}\ov{\mathcal{E}}^{i},0\right].
\end{displaymath}
By \cite[Lemma 3.5]{BurgosFreixasLitcanu:HerStruc}, the quasi-isomorphism $\ov{\mathcal{E}}^{\ast}\dashrightarrow\ov{\mathcal{E}}^{\prime\ast}$ inducing the identity on $\mathcal{F}^{\ast}$ can be lifted to a diagram
\begin{displaymath}
  \xymatrix{
    &\ov{\mathcal{E}}^{''\ast}\ar[ld]_{a}\ar[rd]^{b}	&\\
    \ov{\mathcal{E}}^{\ast}	&	&\ov{\mathcal{E}}^{\prime\ast},
  }
\end{displaymath}
where $a$ and $b$ are quasi-isomorphisms and $\ocone(a)$ is meager \cite[Def. 2.9]{BurgosFreixasLitcanu:HerStruc}. On the one hand, by the characterization \cite[Thm. 2.13]{BurgosFreixasLitcanu:HerStruc} of meager complexes, one can show
\begin{displaymath}
	\left[\sum_{i}(-1)^{i}\ocone(a)^{i},0\right]=0.
\end{displaymath}
On the other hand, we have an exact sequence of complexes
\begin{displaymath}
	0\rightarrow\ov{\mathcal{E}}^{''\ast}\rightarrow\ocone(a)\rightarrow\ov{\mathcal{E}}^{\ast}[1]\rightarrow 0,
\end{displaymath}
whose constituent rows are orthogonally split. This shows
\begin{equation}\label{eq:14}
	\left[\sum_{i}(-1)^{i}\ov{\mathcal{E}}^{'' i},0\right]-\left[\sum_{i}(-1)^{i}\ov{\mathcal{E}}^{i},0\right]=\left[\sum_{i}(-1)^{i}\ocone(a)^{i},0\right]=0.
\end{equation}
Similarly we have
\begin{equation}\label{eq:15}
	\left[\sum_{i}(-1)^{i}\ov{\mathcal{E}}^{'i},0\right]-\left[\sum_{i}(-1)^{i}\ov{\mathcal{E}}^{'' i},0\right]=\left[\sum_{i}(-1)^{i}\ocone(b)^{i},0\right].
\end{equation}
But the complex underlying $\ocone(b)$ is acyclic, so that
\begin{equation}\label{eq:16}
	\left[\sum_{i}(-1)^{i}\ocone(b)^{i},0\right]=[0,c(\widetilde{\ch}(\ocone(b))].
\end{equation}
Finally, by \cite[Def. 3.14, Thm. 4.11]{BurgosFreixasLitcanu:HerStruc} we have
\begin{equation}\label{eq:17}
	\widetilde{\ch}(\Id:\ov{\mathcal{F}}^{\ast}\rightarrow\ov{\mathcal{F}}^{'\ast})=\widetilde{\ch}(\ocone(b)).
\end{equation}
Putting \eqref{eq:14}--\eqref{eq:17} together allows to conclude.
\end{proof}
\begin{notation}
Let $\mathcal{X}$ be a regular arithmetic variety. Then we still denote the image of an object $[\ov{\mathcal{F}}^{\ast},\widetilde{\eta}]\in\Ob\hDb(\mathcal{X},C)$ by the morphism of the proposition by $[\ov{\mathcal{F}}^{\ast},\widetilde{\eta}]$. We call this image the class of $[\ov{\mathcal{F}}^{\ast},\widetilde{\eta}]$ in arithmetic $K$-theory.
\end{notation}
\begin{remark}
Two tightly isomorphic objects in $\hDb(\mathcal{X},C)$ have the same class in arithmetic $K$-theory.
\end{remark}
Let $f:\mathcal{X}\rightarrow\mathcal{Y}$ be a morphism of regular arithmetic varieties and let $C$, $C'$ be $\mathcal{D}_{\as}(X)$ and $\mathcal{D}_{\as}(Y)$ complexes respectively, with a commutative diagram as in \eqref{eq:12}. Then there is a commutative diagram of functors
\begin{displaymath}
	\xymatrix{
		\hDb(\mathcal{Y},C')\ar[r]^{f^{\ast}}\ar[d]	&\hDb(\mathcal{X},C)\ar[d]\\
		\widehat{K}_{0}(\mathcal{Y},C')\ar[r]_{f^{\ast}}	&\widehat{K}_{0}(\mathcal{X},C).
	}
\end{displaymath}
The following statement is a particular case.
\begin{proposition}
Let $f:\mathcal{X}\rightarrow\mathcal{Y}$ be a morphism of regular arithmetic varieties, and $S\subset T_{0}^{\ast}Y_{\CC}$ a closed conical subset  invariant under complex conjugation, disjoint with the normal directions $N_{f}$ of $f_{\CC}$. Then there is a commutative diagram
\begin{displaymath}
	\xymatrix{
		\hDb(\mathcal{Y},\mathcal{D}_{\D,\as}(Y,S))\ar[r]^-{f^{\ast}}\ar[d]	&\hDb(\mathcal{X},\mathcal{D}_{\D,\as}(X, f^{\ast}(S)))\ar[d]\\
		\widehat{K}_{0}(\mathcal{Y},\mathcal{D}_{\D,\as}(Y,S))\ar[r]_-{f^{\ast}}	&\widehat{K}_{0}(\mathcal{X},\mathcal{D}_{\D,\as}(X,f^{\ast}(S))).
	}
\end{displaymath}
\end{proposition}
If $C,C',C''$ are $\mathcal{D}_{\as}(X)$-complexes with a product $C\otimes C'\rightarrow C''$ compatible with the product of $\mathcal{D}_{\as}(X)$ and $\mathcal{X}$ is regular, then there is a commutative diagram of functors
\begin{displaymath}
	\xymatrix{
		\hDb(\mathcal{X},C)\times\hDb(\mathcal{X},C')\ar[r]^-{\otimes} \ar[d]	&\hDb(\mathcal{X},C'')\ar[d]\\
		\widehat{K}_{0}(\mathcal{X},C)\times\widehat{K}_{0}(\mathcal{X},C')\ar[r]	&\widehat{K}_{0}(\mathcal{X},C''),
	}
\end{displaymath}
where the derived tensor product $\otimes$ is defined by
\begin{displaymath}
	[\ov{\mathcal{F}},\eta]\otimes [\ov{\mathcal{G}},\nu]=
    [\ov{\mathcal{F}}\otimes\ov{\mathcal{G}},
    c(\ch(\ov{\mathcal{F}}))\bullet\nu+
    \eta\bullet c'(\ch(\ov{\mathcal{G}}))+\dd_{C}\eta\bullet\nu]
\end{displaymath}
\begin{proposition}
Let $S, S'$ be closed conical subsets of $T_{0}^{\ast}X_{\CC}$ invariant under the action of complex conjugation, with $S\cap (-S')=\emptyset$. Define $T=S\cup S'\cup (S+S')$. If $\mathcal{X}$ is regular, then there is a commutative diagram of functors
\begin{displaymath}
	\xymatrix{
		\hDb(\mathcal{X},\mathcal{D}_{\D,\as}(X,S))\times\hDb(\mathcal{X},\mathcal{D}_{\D,\as}(X,S'))\ar[r]^-{\otimes} \ar[d]	
			&\hDb(\mathcal{X},\mathcal{D}_{\D,\as}(X,T))\ar[d]\\
		\widehat{K}_{0}(\mathcal{X},\mathcal{D}_{\D,\as}(X,S))\times\widehat{K}_{0}(\mathcal{X},\mathcal{D}_{\D,\as}(X,S'))\ar[r]	
		&\widehat{K}_{0}(\mathcal{X},\mathcal{D}_{\D,\as}(X,T)).
	}
\end{displaymath}
Furthermore, it is compatible with pull-back $f^{\ast}$ whenever defined.
\end{proposition}
Finally, the class functor and the arithmetic Chern character on arithmetic $K$ groups, allow to extend it to arithmetic derived categories.
\begin{notation}\label{not:1}
Let $\mathcal{X}$ be a regular arithmetic variety and $C$ a $\mathcal{D}_{\as}(X)$ complex. We denote by
\begin{displaymath}
	\chh:\hDb(\mathcal{X},C)\longrightarrow\bigoplus_{p}\cha^{p}_{\QQ}(\mathcal{X},C)
\end{displaymath}
the arithmetic Chern character on $\hDb(\mathcal{X},C)$, obtained as
the composition of the class functor and $\chh$ on
$\widehat{K}_{0}(\mathcal{X},C)$.
\end{notation}
\section{Arithmetic characteristic classes}\label{section:ArChar}
Let $\mathcal{X}$ be a regular arithmetic variety. From the previous
sections, there exists a natural functor 
\begin{displaymath}
  \oDb(\mathcal{X})\longrightarrow\widehat{K}_{0}(\mathcal{X},\mathcal{D}_{\as}(X))
\end{displaymath}
that factors through $\hDb(\mathcal{X},C)$, and there is a ring isomorphism
\begin{displaymath}
  \chh:\widehat{K}_{0}(\mathcal{X},\mathcal{D}_{\as}(X))_{\QQ}\longrightarrow
  \bigoplus_{p}\cha^{p}_{\QQ}(\mathcal{X},\mathcal{D}_{\as}(X)).
\end{displaymath}
We therefore obtain a functor
\begin{displaymath}
  \chh:\oDb(\mathcal{X})\longrightarrow
  \bigoplus_{p}\cha^{p}_{\QQ}(\mathcal{X},\mathcal{D}_{\as}(X))
\end{displaymath}
automatically satisfying several compatibilities with the operations
in $\oDb(\mathcal{X})$ and distinguished triangles. More generally, in
this section we construct arithmetic characteristic classes attached
to real additive or multiplicative genera. The case of the arithmetic
Todd class will be specially relevant, since it is involved in the
arithmetic Riemann-Roch theorem. Our construction relies on the one
given by Gillet-Soul\'e \cite{GilletSoule:vbhm}.

Let $B$ be a subring of $\RR$ and $\varphi\in B[[x]]$ a real power series, defining an
additive genus. For each hermitian vector bundle $\ov {\mathcal{E}}$, in
\cite{GilletSoule:vbhm}, there is attached a class
$\widehat{\varphi}(\ov{\mathcal{E}})\in \cha_{B}
^{\ast}(\mathcal{X})$. By the isomorphism \cite[Theorem
3.33]{BurgosKramerKuehn:accavb} we obtain a class 
$\widehat{\varphi}(\ov{\mathcal{E}})\in \cha_{B}
^{\ast}(\mathcal{X},\mathcal{D}_{\as}(X))$.

For every finite complex of smooth hermitian vector
bundles $\ov{\mathcal{E}}^{\ast}$, we put
\begin{displaymath}
  \widehat{\varphi}(\ov{\mathcal{E}}^{\ast})=
  \sum_{i}(-1)^{i}\widehat{\varphi}(\ov{\mathcal{E}}^{i})\in
  \bigoplus_{p}\cha^{p}_{B}(\mathcal{X},\mathcal{D}_{\as}(X)).  
\end{displaymath}
Let us now consider an object $\ov{\mathcal{F}}^{\ast}$ in
$\oDb(\mathcal{X})$. We choose an auxiliary quasi-isomor\-phism
$\psi:\mathcal{E}^{\ast}\dashrightarrow\mathcal{F}^{\ast}$, where
$\mathcal{E}^{\ast}$ is a bounded complex of locally free sheaves on
$\mathcal{X}$. This is possible since $\mathcal{X}$ is regular by
assumption. We also fix auxiliary smooth hermitian metrics on the
individual terms $\mathcal{E}^{i}$. We thus obtain an isomorphism
$\ov{\psi}\colon\ov{\mathcal{E}}^{\ast}\dashrightarrow\ov{\mathcal{F}}^{\ast}$
that in general is not tight. The lack of tightness is measured by a
class $[\ov{\psi}_{\CC}]\in\ov{\KA}(X)$, that we simply denote
$[\ov{\psi}]$ \cite[Sec. 3]{BurgosFreixasLitcanu:HerStruc}. Recall
that Bott-Chern secondary classes can be defined at the level of
$\ov{\KA}(X)$ (see \emph{loc. cit.} Sec. 4, and especially the
characterization given in Prop. 4.6). In particular we have a class
\begin{displaymath}
  \widetilde{\varphi}(\ov{\psi}):=\widetilde{\varphi}([\ov{\psi}])
  \in\bigoplus_{p}\widetilde{\mathcal{D}}_{\as}^{2p-1}(X,p).
\end{displaymath}
\begin{lemma}
The class
\begin{equation}\label{eq:22}
  \widehat{\varphi}(\ov{\mathcal{E}}^{\ast})+
  \amap(\widetilde{\varphi}(\ov{\psi}))\in
  \bigoplus_{p}\cha^{p}_{B}(\mathcal{X},\mathcal{D}_{\as}(X))
\end{equation}
depends only on $\ov{\mathcal{F}}^{\ast}$.
\end{lemma}
\begin{proof}
Let
$\psi^{\prime}\colon\mathcal{E}^{\prime\ast}\dashrightarrow\mathcal{F}^{\ast}$
be another finite locally free resolution, and choose arbitrary
metrics on the $\mathcal{E}^{\prime i}$. We can construct a
commutative diagram of complexes in $\Db(\mathcal{X})$
\begin{equation}\label{eq:27}
  \xymatrix{
    &\mathcal{E}^{\prime\prime\ast}\ar[ld]_{\alpha}\ar[rd]^{\beta}
    &\\
    \mathcal{E}^{\ast}\ar@{-->}[d]_{\psi}	&
    &\mathcal{E}^{\prime\ast}\ar@{-->}[d]^{\psi^{\prime}}\\ 
    \mathcal{F}^{\ast}\ar[rr]^{\Id}	&	&\mathcal{F}^{\ast}, 
  }
\end{equation}
where $\mathcal{E}^{\prime\prime}$ is also a finite complex of locally
free sheaves, that we endow with smooth hermitian metrics, and
$\alpha$, $\beta$ are quasi-isomorphisms. Because the exact sequences
\begin{align*}
	&0\rightarrow\ov{\mathcal{E}}^{\prime\ast}\rightarrow\ocone(\alpha)\rightarrow\ov{\mathcal{E}}^{\prime\prime\ast}[1]\rightarrow 0,\\
	&0\rightarrow\ov{\mathcal{E}}^{\ast}\rightarrow\ocone(\beta)\rightarrow\ov{\mathcal{E}}^{\prime\prime\ast}[1]\rightarrow 0
\end{align*}
have orhtogonally split constituent rows, we find
\begin{align}
	&\widehat{\varphi}(\ocone(\alpha))=\widehat{\varphi}(\ov{\mathcal{E}}^{\ast})-\widehat{\varphi}(\ov{\mathcal{E}}^{\prime\prime\ast}),\\ \label{eq:23}
	& \widehat{\varphi}(\ocone(\beta))=\widehat{\varphi}(\ov{\mathcal{E}}^{\prime\ast})-\widehat{\varphi}(\ov{\mathcal{E}}^{\prime\prime\ast}).
\end{align}
Also, the complexes $\cone(\alpha)$ and $\cone(\beta)$ are acyclic, so that
\begin{align}
	&\widehat{\varphi}(\ocone(\alpha))=\amap(\widetilde{\varphi}(\ocone(\alpha)))=\amap(\widetilde{\varphi}(\ov{\alpha})),\\
	&\widehat{\varphi}(\ocone(\beta))=\amap(\widetilde{\varphi}(\ocone(\beta)))=\amap(\widetilde{\varphi}(\ov{\beta})), \label{eq:26}
\end{align}
where we took into account the very definition of the class of an isomorphism in $\oDb(X)$. From the relations \eqref{eq:23}--\eqref{eq:26} we derive
\begin{equation}\label{eq:28}
	\widehat{\varphi}(\ov{\mathcal{E}}^{\ast})-\widehat{\varphi}(\ov{\mathcal{E}}^{\prime\ast})
	=\amap(\widetilde{\varphi}(\ov{\alpha})-\widetilde{\varphi}(\ov{\beta}))=\amap(\widetilde{\varphi}(\ov{\alpha}\circ\ov{\beta}^{-1})),
\end{equation}
where we plugged $\widetilde{\varphi}(\ov{\alpha}\circ\ov{\beta}^{-1})=\widetilde{\varphi}(\ov{\alpha})-\widetilde{\varphi}(\ov{\beta})$ \cite[Prop. 4.13]{BurgosFreixasLitcanu:HerStruc}. But by diagram \eqref{eq:27} we have $\ov{\alpha}\circ\ov{\beta}^{-1}=\ov{\psi}^{-1}\circ\ov{\psi}^{\prime}$.  This fact combined with \eqref{eq:28} implies
\begin{displaymath}
	\begin{split}
	\widehat{\varphi}(\ov{\mathcal{E}}^{\ast})-\widehat{\varphi}(\ov{\mathcal{E}}^{\prime\ast})
	=&\amap(\widetilde{\varphi}(\ov{\alpha})-\widetilde{\varphi}(\ov{\beta}))\\
	&=\amap(\widetilde{\varphi}(\ov{\psi}^{-1}\circ\ov{\psi}^{\prime}))\\
	&\hspace{0.4cm}=\amap(\widetilde{\varphi}(\ov{\psi}^{\prime}))-\amap(\widetilde{\varphi}(\ov{\psi})).
	\end{split}
\end{displaymath}
This completes the proof of the lemma.
\end{proof}
\begin{definition}
The notations being as above, we define
\begin{displaymath}
  \widehat{\varphi}(\ov{\mathcal{F}}^{\ast}):=
  \widehat{\varphi}(\ov{\mathcal{E}}^{\ast})+\amap(\widetilde{\varphi}(\ov{\psi}))
  \in\bigoplus_{p}\cha^{p}_{B}(\mathcal{X},\mathcal{D}_{\as}(X)),
\end{displaymath}
\end{definition}
The additive arithmetic characteristic classes are indeed additive with respect to direct sum, and are compatible with pull-back by morphisms of arithmetic varieties. However, the most important property of additive arithmetic characteristic classes is the behavior with respect to distinguished triangles. The reader may review \cite[Def. 3.29, Thm. 3.33, Def. 4.17, Thm. 4.18]{BurgosFreixasLitcanu:HerStruc} for definitions and main properties, especially for the class in $\ov{\KA}(X)$ and secondary class of a distinguished triangle.
\begin{theorem}\label{thm:3}
Let us consider a distinguished triangle in $\oDb(\mathcal{X})$:
\begin{displaymath}
	\ov{\tau}\colon\quad \ov{\mathcal{F}}^{\ast}\dashrightarrow\ov{\mathcal{G}}^{\ast}\dashrightarrow\ov{\mathcal{H}}^{\ast}
	\dashrightarrow\ov{\mathcal{F}}^{\ast}_{0}[1].
\end{displaymath}
Then we have
\begin{displaymath}
	\widehat{\varphi}(\mathcal{F}^{\ast})-\widehat{\varphi}(\mathcal{G}^{\ast})+\widehat{\varphi}(\mathcal{H}^{\ast})
	=\amap(\widetilde{\varphi}(\ov{\tau}))
\end{displaymath}
in $\bigoplus_{p}\cha^{p}_{B}(\mathcal{X},\mathcal{D}_{\as}(X))$. In particular, if $\ov{\tau}$ is tightly distinguished, we have
\begin{displaymath}
	\widehat{\varphi}(\ov{\mathcal{G}}^{\ast})=\widehat{\varphi}(\ov{\mathcal{F}}^{\ast})+\widehat{\varphi}(\ov{\mathcal{H}}^{\ast}).
\end{displaymath}
\end{theorem}
\begin{proof}
It is possible to find a diagram
\begin{displaymath}
	\xymatrix{
		\ov{\eta}\colon  &\mathcal{E}^{\ast}\ar@{-->}[d]^{f}\ar[r]^{\alpha}		&\mathcal{E}^{\prime\ast}\ar@{-->}[d]^{g}\ar[r]	&\cone(\alpha)\ar@{-->}[d]^{h}\ar[r]	&\mathcal{E}^{\ast}[1]\ar@{-->}[d]^{f[1]}\\
		\ov{\tau}\colon	&\mathcal{F}^{\ast}\ar@{-->}[r]	&\mathcal{G}^{\ast}\ar@{-->}[r]	&\mathcal{H}^{\ast}\ar@{-->}[r]	&\mathcal{F}^{\ast}[1],
	}
\end{displaymath}
where $\mathcal{E}^{\ast}$, $\mathcal{E}^{\prime\ast}$ are bounded complexes of locally free sheaves and the vertical arrows are isomorphisms in $\Db(\mathcal{X})$. We choose arbitrary smooth hermitian metrics on the $\mathcal{E}^{i}$, $\mathcal{E}^{\prime j}$, and put the orthogonal sum metric on $\cone(\alpha)$. Then, by construction of the arithmetic characteristic classes, we have
\begin{displaymath}
	\begin{split}
		\widehat{\varphi}(\ov{\mathcal{F}}^{\ast})-\widehat{\varphi}(\ov{\mathcal{G}}^{\ast})
		+\widehat{\varphi}(\ov{\mathcal{H}}^{\ast})=&
		\widehat{\varphi}(\ov{\mathcal{E}}^{\ast})-\widehat{\varphi}(\ov{\mathcal{E}}^{\prime\ast})
		+\widehat{\varphi}(\ocone(\alpha))\\
		&+a(\widetilde{\varphi}(\ov{f})-\widetilde{\varphi}(\ov{g})+\widetilde{\varphi}(\ov{h})).
	\end{split}
\end{displaymath}
Because the exact sequence
\begin{displaymath}
	0\rightarrow\ov{\mathcal{E}}^{\prime\ast}\rightarrow\ocone(\alpha)\rightarrow\ov{\mathcal{E}}^{\ast}[1]\rightarrow 0
\end{displaymath}
has orthogonally split constituent rows, we observe
\begin{displaymath}
	\widehat{\varphi}(\ov{\mathcal{E}}^{\ast})-\widehat{\varphi}(\ov{\mathcal{E}}^{\prime\ast})
		+\widehat{\varphi}(\ocone(\alpha))=0.
\end{displaymath}
Moreover, by \cite[Thm 3.33 (vii)]{BurgosFreixasLitcanu:HerStruc} the equality
\begin{displaymath}
	\widetilde{\varphi}(\ov{f})-\widetilde{\varphi}(\ov{g})+\widetilde{\varphi}(\ov{h})
	=\widetilde{\varphi}(\ov{\tau})-\widetilde{\varphi}(\ov{\eta}),
\end{displaymath}
holds, and $\widetilde{\varphi}(\ov{\eta})=0$ since $\ov{\eta}$ is tightly distinguished. The theorem now follows.
\end{proof}
We may also say a few words on multiplicative arithmetic
characteristic classes. We follow the discussion in
\cite[Sec. 5]{BurgosFreixasLitcanu:HerStruc}. Assume from now on that $\QQ\subset B$. Let $\psi\in B[[x]]$ be a formal power series with 
$\psi^{0}=1$. We denote also by $\psi$ the associated multiplicative
genus. 
Then
\begin{displaymath}
	\varphi=\log(\psi)\in B[[x]]
\end{displaymath}
defines an additive genus, to which we can associate an addtive
arithmetic characteristic genus $\widehat{\varphi}$. Then, given an
object $\ov{\mathcal{F}}^{\ast}$ in $\oDb(\mathcal{X})$, we have a
well defined class 
\begin{displaymath}
  \widehat{\psi}_{m}(\ov{\mathcal{F}}^{\ast}):=\exp(\widehat{\varphi}(\ov{\mathcal{F}}^{\ast}))
  \in\bigoplus_{p}\cha^{p}_{B}(\mathcal{X},\mathcal{D}_{\as}(X)).
\end{displaymath}
Observe this construction uses the ring structure of
$\bigoplus_{p}\cha^{p}_{B}(\mathcal{X},\mathcal{D}_{\as}(X))$, hence
the regularity of $\mathcal{X}$ and the product structure of
$\mathcal{D}_{\as}(X)$. In case $\psi$ has rational coefficients,
$\widehat{\psi}(\ov{\mathcal{F}}^{\ast})$ takes values in  
$\bigoplus_{p}\cha^{p}_{\QQ}(\mathcal{X},\mathcal{D}_{\as}(X))$. We
also recall that there is a multiplicative secondary class associated
to $\psi$, denoted $\widetilde{\psi}_{m}$, that can be expressed in
terms of $\widetilde{\varphi}$:
\begin{displaymath}
  \widetilde{\psi}_{m}(\theta)=\frac{\exp(\varphi(\theta))-1}{\varphi(\theta)}\widetilde{\varphi}(\theta),
\end{displaymath}
where $\theta$ is any class in $\ov{\KA}(X)$. If $\ov{\tau}$ is a distinguished triangle in $\oDb(\mathcal{X})$, we will simply write $\widetilde{\psi}_{m}(\ov{\tau})$ instead of $\widetilde{\psi}_{m}([\ov{\tau}])$.

\begin{theorem}\label{thm:4}
Let $\psi$ be a multiplicative genus, with degree 0 component
$\psi^{0}=1$. Then, for every distinguished triangle in
$\oDb(\mathcal{X})$ 
\begin{displaymath}
	\ov{\tau}\colon	\ov{\mathcal{F}}^{\ast}\dashrightarrow\ov{\mathcal{G}}^{\ast}\dashrightarrow\ov{\mathcal{H}}^{\ast}
	\dashrightarrow\ov{\mathcal{F}}[1]
\end{displaymath}
we have the relation
\begin{displaymath}
	\widehat{\psi}(\ov{\mathcal{F}}^{\ast})^{-1}\widehat{\psi}(\ov{\mathcal{G}})\widehat{\psi}(\ov{\mathcal{H}})^{-1}-1=
	\amap(\widetilde{\psi}_{m}(\ov{\tau})).
\end{displaymath}
In particular, if $\ov{\tau}$ is tightly distinguished, the equality
\begin{displaymath}
	\widehat{\psi}(\ov{\mathcal{G}}^{\ast})=\widehat{\psi}(\ov{\mathcal{F}}^{\ast})\widehat{\psi}(\ov{\mathcal{H}}^{\ast})
\end{displaymath}
holds.
\end{theorem}
\begin{proof}
It is enough to exponentiate the relation provided by Theorem \ref{thm:3} and observe that, due to the mutliplicative law in $\bigoplus_{p}\cha^{p}(\mathcal{X},\mathcal{D}_{\as}(X))$, one has
\begin{displaymath}
	\exp(\amap(\widetilde{\varphi}(\ov{\tau})))=1+\amap\left(\frac{\exp(\varphi(\ov{\tau}))-1}{\exp(\varphi(\ov{\tau}))}\widetilde{\varphi}(\ov{\tau})\right).
\end{displaymath}
\end{proof}
\begin{example}
\begin{enumerate}
\item The arithmetic Chern class, is the additive class
  attached to the additive genus $\ch(x)=e^{x}$. It coincides
  with the character $\chh$ of the previous section. 
\item The arithmetic Todd class, is the multiplicative class attached to the genus
  \begin{displaymath}
    \Td(x)=\frac{x}{1-e^{-x}}.
  \end{displaymath}
  Observe that the formal series of $\Td(x)$ has constant coefficient 1.
\end{enumerate}
\end{example}
\begin{remark}
If $C$ is a $\mathcal{D}_{\as}(X)$-complex, we can define arithmetic
characteristic classes with values in
$\bigoplus_{p}\cha^{p}_{B}(\mathcal{X},C)$, just taking their image by the
natural morphism
$\bigoplus_{p}\cha^{p}_{B}(\mathcal{X},\mathcal{D}_{\as}(X))\rightarrow
\bigoplus_{p}\cha^{p}_{B}(\mathcal{X},C)$. We will use the same
notations to refer to these classes.
\end{remark}

\section{Direct images and generalized analytic torsion}\label{section:DirectImage}
In the preceding section we introduced the arithmetic $K$-groups and
the arithmetic derived categories. They satisfy elementary
functoriality properties and are related by a class functor. A missing
functoriality is the push-forward by \emph{arbitrary} projectives
morphisms of arithmetic varieties. Similarly to the work of
Gillet-R\"ossler-Soul\'e
\cite{GilletRoesslerSoule:_arith_rieman_roch_theor_in_higher_degrees},
we will define direct images after choosing a generalized analytic
torsion theory in the sense of
\cite{BurgosFreixasLitcanu:GenAnTor}. Our theory is more general in
that we don't require our morphisms to be smooth over the generic
fiber. At the archimedean places, we are thus forced to work with
complexes of currents with controlled wave front sets. In this level
of generality, the theory of arithmetic Chow groups, arithmetic
$K$-theory and arithmetic derived categories has already been
discussed. Let us recall that a generalized analytic torsion theory is
not unique, but according to \cite[Thm. 7.7 and
Thm. 7.14]{BurgosFreixasLitcanu:GenAnTor} it is classified by a real
additive genus.

Our theory of generalized analytic torsion classes involves the notion of relative metrized complex \cite[Def. 2.5]{BurgosFreixasLitcanu:GenAnTor}. In the sequel we will need a variant on real smooth quasi-projective schemes $X=(X_{\CC},F_{\infty})$. With respect to \emph{loc. cit.}, this amounts to imposing an additional invariance under the action of $F_{\infty}$.

\begin{definition}
A real relative metrized complex is a triple $\ov{\xi}=(\ov{f}, \ov{\mathcal{F}}^{\ast},\ov{f_{\ast}\mathcal{F}^{\ast}})$, where
\begin{itemize}
	\item $\ov{f}:X\rightarrow Y$ is a projective morphism of real smooth quasi-projective varieties, together with a hermitian structure on the tangent complex $T_{f}$, invariant under the action of complex conjugation;
	\item $\ov{\mathcal{F}}^{\ast}$ is an object in $\oDb(X)$;
	\item $\ov{f_{\ast}\mathcal{F}}^{\ast}$ is an object in $\oDb(Y)$ lying over $f_{\ast}\mathcal{F}^{\ast}$.
\end{itemize}
\end{definition}
The following lemma is checked by a careful proof reading of the construction of generalized analytic torsion classes in \cite{BurgosFreixasLitcanu:GenAnTor}.
\begin{lemma}
Let $T$ be a theory of generalized analytic torsion classes. Then, for every real relative metrized complex $\ov{\xi}=(\ov{f}\colon X\to Y,\ov{\mathcal{F}}^{\ast},\ov{f_{\ast}\mathcal{F}}^{\ast})$, $T(\ov{\xi})$ is a class of real currents,
\begin{displaymath}
	T(\ov{\xi})\in\bigoplus_{p}\widetilde{\mathcal{D}}^{2p-1}_{\D,\as}(X,N_{f},p),
\end{displaymath}
where $N_{f}$ is the cone of normal directions to $f$.
\end{lemma}
It will be useful to have an adaptation of the category
$\ov{\Sm}_{\ast/\CC}$ \cite[Sec. 5]{BurgosFreixasLitcanu:HerStruc} to
real quasi-projective schemes, that we denote
$\ov{\Sm}_{\ast/\RR}$. In this category, the objects are real smooth
quasi-projective schemes, and the morphisms are projective morphisms
with a hermitian structure invariant under complex conjugation. The
composition law is then described in \emph{loc. cit.}, Def. 5.7, with
the help of the hermitian cone construction. We will follow the
notation introduced in
\cite[Def. 2.12]{BurgosFreixasLitcanu:GenAnTor}. 
\begin{notation}
  Let $\ov{f}\colon X\to Y$ be in $\ov{\Sm}_{\ast/\RR}$, of pure
  relative dimension $e$, $C$ a $\mathcal{D}_{\as}(X)$-complex, $C'$ a
  $\mathcal{D}_{\as}(Y)$-complex and $f_{\ast}:C\rightarrow C'$ a
  morphism of fitting into the commutative diagram
  \eqref{eq:18}. Assume furthermore that $C$ is a
  $\mathcal{D}_{\as}(X)$-module with a product law $\bullet$ as in
  \eqref{eq:19}. Then we put
\begin{displaymath}
  \begin{split}
    \ov{f}_{\ttwist}:C^{\ast}(\ast)&\longrightarrow C^{\prime\ast-2e}(p-e)\\
    \eta&\longmapsto f_{\ast}(\eta\bullet \Td(T_{\ov{f}})).
  \end{split}
\end{displaymath}
This morphism induces a corresponding morphism on
$\widetilde{C}^{\ast}(\ast)$, for which we use the same notation. 
\end{notation}
For an arithmetic ring $A$, we introduce $\ov{\Reg}_{\ast/A}$ the
category of  quasi-proj\-ective regular arithmetic varieties
over $A$, with projective morphisms endowed (at the archimedean
places) with a hermitian structure invariant under complex
conjugation. By construction, there is a natural base change functor 
\begin{displaymath}
	\ov{\Reg}_{\ast/A}\longrightarrow\ov{\Sm}_{\ast/\RR}.
\end{displaymath}
Given $\ov{f}\colon\mathcal{X}\to\mathcal{Y}$ a morphism in $\ov{\Reg}_{\ast/A}$ and objects $\ov{\mathcal{F}}^{\ast}$, $\ov{f_{\ast}\mathcal{F}}^{\ast}$ in $\oDb(\mathcal{X})$ and $\oDb(\mathcal{Y})$, respectively, we may consider the corresponding real relative metrized complex, that we will abusively write $\ov{\xi}=(\ov{f},\ov{\mathcal{F}},\ov{f_{\ast}\mathcal{F}}^{\ast})$, and its analytic torsion class $T(\ov{\xi})$. We may also write $\ov{f}_{\ttwist}$ instead of $\ov{f}_{\CC\,\ttwist}$, etc. 

We are now in position to construct the arithmetic counterpart of \cite[Eq. (10.6)]{BurgosFreixasLitcanu:GenAnTor}, namely the direct image functor on arithmetic derived categories, as well as a similar push-forward on arithmetic $K$-theory.
\begin{definition}
Let $\ov{f}\colon\mathcal{X}\to\mathcal{Y}$ be a morphism in
$\ov{\Reg}_{\ast/A}$, $C$ a $\mathcal{D}_{\as}(X)$-complex, $C'$ a
  $\mathcal{D}_{\as}(Y)$-complex, both satisfying the
hypothesis (H1) and (H2), and $f_{\ast}:C\rightarrow C'$ a
  morphism fitting into a commutative diagram like
  \eqref{eq:18}.
  \begin{enumerate}
  \item  We define the functor
\begin{displaymath}
	\ov{f}_{\ast}\colon\hDb(\mathcal{X},C)\longrightarrow\hDb(\mathcal{Y},C'),
\end{displaymath}
acting on objects by the assignment
\begin{displaymath}
	[\ov{\mathcal{F}}^{\ast},\widetilde{\eta}]\longmapsto [\ov{f_{\ast}\mathcal{F}}^{\ast},\ov{f}_{\ttwist}(\widetilde{\eta})-c'(T(\ov{f},\ov{\mathcal{F}}^{\ast},\ov{f_{\ast}\mathcal{F}}^{\ast}))].
\end{displaymath}
Here $\ov{f_{\ast}\mathcal{F}}^{\ast}$ carries an arbitrary choice of hermitian structure. The action on morphisms of $\ov{f}_{\ast}$ is just the usual action $f_{\ast}$ on morphisms of $\Db(\mathcal{X})$.

\item  We define a morphism of groups
\begin{displaymath}
	\begin{split}
		f_{\ast}:\widehat{K}_{0}(\mathcal{X},C)&\longrightarrow\widehat{K}_{0}(\mathcal{X},C')\\
		[\ov{\mathcal{E}},\widetilde{\eta}]&\longmapsto [\sum_{i}(-1)^{i}\ov{\mathcal{E}}^{\prime i},\ov{f}_{\ttwist}(\widetilde{\eta})-T(\ov{f},\ov{\mathcal{E}},\ov{f_{\ast}\mathcal{E}}^{\ast})],
	\end{split}
\end{displaymath}
where we choose an arbitrary quasi-isomorphism
$\mathcal{E}^{\prime\ast}\dashrightarrow f_{\ast}\mathcal{E}$ and
arbitrary smooth hermitian metrics on the $\mathcal{E}^{\prime
  i}$. Here $f_{\ast}\mathcal{E}$ denotes the derived direct image of
the single locally free sheaf $\mathcal{E}$.

  \end{enumerate}
\end{definition}
Notice the previous definition makes sense by the anomaly formulas satisfied by analytic torsion theories \cite[Prop. 7.4]{BurgosFreixasLitcanu:GenAnTor}. 
Both push-forwards are compatible through the class map from arithmetic derived categories to arithmetic $K$-theory.
\begin{theorem}\label{thm:6}
There is a commutative diagram of functors
\begin{displaymath}
	\xymatrix{
		\hDb(\mathcal{X},C)\ar[r]^-{\ov{f}_{\ast}}\ar[d]	&\hDb(\mathcal{Y},C')\ar[d]\\
		\widehat{K}_{0}(\mathcal{X},C)\ar[r]	_-{\ov{f}_{\ast}}	&\widehat{K}_{0}(\mathcal{Y},C').
	}
\end{displaymath}
This is in particular true for $C=\mathcal{D}_{\D,\as}(X,S)$ and $C'=\mathcal{D}_{\D,\as}(Y,f_{\ast}(S))$, where $S\subset T_{0}^{\ast}X$ is a closed conical subset invariant under the action of complex conjugation.
\end{theorem}
\begin{proof}
Let $[\ov{\mathcal{F}}^{\ast},\widetilde{\eta}]$ be an object in
$\hDb(\mathcal{X},C)$. By Remark \ref{rem:1}, we can suppose the
hermitian structure on $\ov{\mathcal{F}}^{\ast}$ is given by a
quasi-isomorphism
$\ov{\mathcal{E}}^{\ast}\dashrightarrow\mathcal{F}^{\ast}$, where
$\ov{\mathcal{E}}^{\ast}$ is a finite complex of locally free sheaves,
each one endowed with a smooth hermitian metric. For every $i$, we
have to endow $f_{\ast}\mathcal{E}^{i}$ with a hermitian structure
given by a finite locally free resolution
$\ov{\mathcal{E}}^{i\ast}\dashrightarrow f_{\ast}\mathcal{E}^{i}$ and
a choice of arbitrary smooth hermitian metric on every piece
$\mathcal{E}^{ij}$. From the data
$\ov{\mathcal{E}}^{i\ast}\dashrightarrow f_{\ast}\mathcal{E}^{i}$,
every $i$, the procedure of
\cite[Def. 3.39]{BurgosFreixasLitcanu:HerStruc} produces a hermitian
structure on $f_{\ast}\mathcal{E}^{\ast}$, via the hermitian cone
construction. Observe the construction of \emph{loc. cit.} can be done
in $\oDb(\mathcal{Y})$. Combined with the fixed quasi-isomorphism
$\mathcal{E}^{\ast}\dashrightarrow\mathcal{F}^{\ast}$, we thus obtain
a hermitian structure on $f_{\ast}\mathcal{F}^{\ast}$, that we denote
by $\ov{f_{\ast}\mathcal{F}}^{\ast}$. 

The class of $[\ov{\mathcal{F}},\widetilde{\eta}]$ in $\hDb(\mathcal{X},C)$ is thus
\begin{displaymath}
	\sum_{i}(-1)^{i}[\ov{\mathcal{E}}^{i},0] + [0,\widetilde{\eta}].
\end{displaymath}
Its image under $\ov{f}_{\ast}$ is
\begin{equation}\label{eq:20}
	\sum_{i}(-1)^{i}\left((\sum_{j}(-1)^{j}[\ov{\mathcal{E}}^{ij},0])-[0,c'(T(\ov{f},\ov{\mathcal{E}}^{i}, \ov{f_{\ast}\mathcal{E}}^{i}))]\right) + [0,\ov{f}_{\ttwist}(\widetilde{\eta})].
\end{equation}
The class of $\ov{f}_{\ast}[\ov{\mathcal{F}},\widetilde{\eta}]$ in $\hDb(\mathcal{Y},C')$ is
\begin{equation}\label{eq:21}
	[\ov{f_{\ast}\mathcal{F}}^{\ast},\ov{f}_{\ttwist}(\widetilde{\eta})-c'(T(\ov{f},\ov{\mathcal{F}},\ov{f_{\ast}\mathcal{F}}^{\ast}))].
\end{equation}
By the choice of the hermitian structures on $\ov{\mathcal{F}}^{\ast}$ and $\ov{f_{\ast}\mathcal{F}}^{\ast}$ and by Theorem \cite[Prop. 7.6]{BurgosFreixasLitcanu:GenAnTor}, we have
\begin{displaymath}
	T(\ov{f},\ov{\mathcal{F}},\ov{f_{\ast}\mathcal{F}}^{\ast})=\sum_{i}(-1)^{i} T(\ov{f},\ov{\mathcal{E}}^{i},\ov{f_{\ast}\mathcal{E}}^{i}).
\end{displaymath}
Therefore the class in arithmetic $K$-theory of \eqref{eq:21} equals \eqref{eq:20}.
\end{proof}
There are several compatibilities between direct images, inverse images and derived tensor product. We now state them without proof, referring the reader to \cite[Thm. 10.7]{BurgosFreixasLitcanu:GenAnTor} for the details.
\begin{proposition}\label{prop:5}
 Let $\ov{f}\colon\mathcal{X}\to\mathcal{Y}$ and $\ov{g}\colon\mathcal{Y}\to \mathcal{Z}$ be morphisms in $\ov{\Reg}_{\ast/A}$. Let $S\subset T^{\ast}X_{0}$ and $T\subset T^{\ast}Y_{0}$ be closed conical subsets.
\begin{enumerate}
	\item (Functoriality of push-forward) We have the relation
  		\begin{displaymath}
    			(\ov{g}\circ \ov{f})_{\ast}=\ov{g}_{\ast}\circ\ov{f}_{\ast},
   		\end{displaymath}
		as functors $\hDb(\mathcal{X},S)\rightarrow\hDb(\mathcal{Z},g_{\ast}f_{\ast}S)$.
  \item (Projection formula) Assume that $T\cap N_{f}=\emptyset$ and that $T+f_{\ast}S$ does not cross the zero section of $T^{\ast}Y$. Let $[\ov{\mathcal{F}}^{\ast},\widetilde{\eta}]$ be in $\hDb(\mathcal{X},S)$ and $[\ov{\mathcal{F}}^{\prime\ast},\widetilde{\eta}^{\prime}]$ in $\hDb(\mathcal{Y},T)$. Then
	\begin{displaymath}
     		\ov{f}_{\ast}([\ov{\mathcal{F}}^{\ast},\widetilde{\eta}] \otimes
      		f^{\ast}[\ov{\mathcal{F}}^{\prime\ast},\widetilde{\eta}^{\prime}])
     		= \ov{f}_{\ast}[\ov{\mathcal{F}},\widetilde{\eta}]\otimes [\ov{\mathcal{F}}^{\prime},\widetilde{\eta}^{\prime}]
    \end{displaymath}
    in $\hDb(\mathcal{Y},W)$, where
    \begin{displaymath}
    	W=f_{\ast}(S + f^{\ast}T)\cup f_{\ast}S\cup f_{\ast}f^{\ast}T.
    \end{displaymath}
	\item There are analogous relations on the level of arithmetic $K$-theory, compatible with the class functor from arithmetic derived categories.
\end{enumerate}
\end{proposition}
\begin{remark}
Strictly speaking, the equalities provided by the previous statement should be canonical isomorphisms, but as usual we abuse the notations and pretend they are equalities.
\end{remark}
\section{Arithmetic Grothendieck-Riemann-Roch}\label{section:ARR}
\subsection{Statement and reductions}
\paragraph{Hermitian tangent complexes.} Let $\ov{f}\colon\mathcal{X}\rightarrow\mathcal{Y}$ be a morphism in $\ov{\Reg}_{\ast/A}$. We explain how to construct the associated hermitian tangent complex $T_{\ov{f}}$. This is an object in $\oDb(\mathcal{X})$, well defined up to tight isomorphism. 

Because $\mathcal{X}$, $\mathcal{Y}$ are regular schemes and $f$ is projective, it is automatically a l.c.i. morphism. The tangent complex of $f$ is an object in $\Db(\mathcal{X})$, well defined up to isomorphism. Consider a factorization
\begin{displaymath}
	\xymatrix{
		\mathcal{X}\ar@{^{(}->}[r]^-{i}\ar[rd]	&\mathcal{Z}\ar[d]^{\pi}\\
		&\mathcal{Y},
	}
\end{displaymath}
with $i$ being a closed regular immersion and $\pi$ a smooth morphism. For instance, one may choose $\mathcal{Z}=\PP^{n}_{\mathcal{Y}}$, for some $n$. Let us denote by $\mathcal{I}$ the ideal defining the closed immersion $i$. Then $\mathcal{I}/\mathcal{I}^{2}$ is a locally free sheaf on $\mathcal{X}$ and, as customary, we define the normal bundle $N_{\mathcal{X}/\mathcal{Z}}=(\mathcal{I}/\mathcal{I}^{2})^{\vee}$. There is a morphism of coherent sheaves
\begin{displaymath}
	\varphi\colon i^{\ast}T_{\mathcal{Z}/\mathcal{Y}}\longrightarrow N_{\mathcal{X}/\mathcal{Z}},
\end{displaymath}
namely the dual of the differential map $d:\mathcal{I}/\mathcal{I}^{2}\to i^{\ast}\Omega_{\mathcal{Z}/\mathcal{Y}}$. We consider $T_{\mathcal{Z}/\mathcal{Y}}$ as a complex concentrated in degree 0, and $N_{\mathcal{X}/\mathcal{Z}}$ as a complex concentrated in degree one. We then put
\begin{displaymath}
	T_{f}:=\cone(\varphi)[-1].
\end{displaymath}
The isomorphism class of $T_{f}$ in $\Db(\mathcal{X})$ is independent of the factorization. 

The base change to $\CC$ of $T_{f}$ is naturally isomorphic to the tangent complex $T_{f_{\CC}}:TX_{\CC}\rightarrow f_{\CC}^{\ast}Y_{\CC}$, which is equipped with a hermitian structure by assumption. Therefore, the data provided by the constructed complex $T_{f}$ and the hermitian structure on $\ov{f}$ determine an object $T_{\ov{f}}$ in $\oDb(\mathcal{X})$, which is well defined up to tight isomorphism. By Theorem \ref{thm:4}, the arithmetic Todd class of $T_{\ov{f}}$ is unambiguously defined.
\begin{definition}
The arithmetic Todd class of $\ov{f}$ is
\begin{displaymath}
  \widehat{\Td}(\ov{f}):=\widehat{\Td}(T_{\ov{f}})\in
  \bigoplus_{p}\cha^{p}_{\QQ}(\mathcal{X},\mathcal{D}_{\as}(X)). 
\end{displaymath}
\end{definition}
\begin{theorem}
Let $\ov{f}\colon\mathcal{X}\rightarrow\mathcal{Y}$,
$\ov{g}\colon\mathcal{Y}\rightarrow\mathcal{Z}$ be morphisms in
$\ov{\Reg}_{\ast/A}$. Then we have an equality 
\begin{displaymath}
  \widehat{\Td}(\ov{g}\circ\ov{f})=f^{\ast}\widehat{\Td}(\ov{g})\cdot\widehat{\Td}(\ov{f})
\end{displaymath}
in $\bigoplus_{p}\cha^{p}_{\QQ}(\mathcal{X},\mathcal{D}_{\as}(X))$.
\end{theorem}
\begin{proof}
By construction of $T_{\ov{f}}$ and definition of the composition rule of morphisms in $\ov{\Reg}_{\ast/A}$, there a is tightly distinguished triangle in $\oDb(\mathcal{X})$
 \begin{displaymath}
	T_{\ov{f}}\dra T_{\ov{g}\circ \ov{f}}\dra f^{\ast}T_{\ov{g}}\dra T_{\ov{f}}[1].
\end{displaymath}
We conclude by an application of Theorem \ref{thm:4}.
\end{proof}
\paragraph{Statement.} The arithmetic Grothendieck-Riemann-Roch theorem describes the behavior of the arithmetic Chern character with respect to the push-forward functor. Recall that the definition of the push-forward functor depends on the choice of a theory of generalized analytic torsion classes. In its turn, such a theory corresponds to a real additive genus. 
\begin{theorem}\label{thm:5}
Let $\ov{f}\colon\mathcal{X}\rightarrow\mathcal{Y}$ be a morphism in $\ov{\Reg}_{\ast/A}$. Fix a closed conical subset $W$ of $T_{0}^{\ast}X$ and a theory of generalized analytic torsion classes $T$, whose associated real additive genus is $S$. Then, the derived direct image functor $$\ov{f}_{\ast}\colon\hDb(\mathcal{X},\mathcal{D}_{\D,\as}(X,W))\longrightarrow\hDb(\mathcal{Y},\mathcal{D}_{\D,\as}(Y,f_{\ast}(W)))$$ attached to $T$ satisfies the equation
\begin{equation}\label{eq:29}
	\chh(\ov{f}_{\ast}\alpha)=f_{\ast}(\chh(\alpha)\widehat{\Td}(\ov{f}))-\amap(f_{\ast}(\ch(\mathcal{F}^{\ast}_{\CC})\Td(T_{f_{\CC}})S(T_{f_{\CC}})),
\end{equation}
for every object
$\alpha=[\ov{\mathcal{F}}^{\ast},\widetilde{\eta}]$. The equality
takes place in the arithmetic Chow group
$\bigoplus_{p}\cha^{p}_{\QQ}(\mathcal{Y},\mathcal{D}_{\D,\as}(Y,f_{\ast}(W)))$.
\end{theorem}
Recall that the genus $S$=0 corresponds to the homogeneous theory
$T^{h}$, which is characterized by satisfying an additional
homogeneity property \cite[Sec. 9]{BurgosLitcanu:SingularBC}. Roughly
speaking, this condition is exactly the one guaranteeing an
arithmetic Grothendieck-Riemann-Roch theorem without
correction term.
\begin{corollary}
The direct image functor $\ov{f}^{h}_{\ast}$ attached to the
homogeneous generalized analytic torsion theory $T^{h}$ satisfies an
exact Grothendieck-Riemann-Roch type formula:
\begin{displaymath}
  \chh(\ov{f}^{h}_{\ast}\alpha)=f_{\ast}(\chh(\alpha)\widehat{\Td}(\ov{f})).
\end{displaymath}
\end{corollary}
A Grothendieck-Riemann-Roch type theorem in Arakelov geometry was
first proven by Gillet-Soul\'e \cite{GilletSoule:aRRt}, for the degree
1 part of the Chern character (namely the determinan of the
cohomology) and under the restriction on the morphism $f$ to be smooth
over $\CC$. Also they can only deal with hermitian vector
bundles. They used the holomorphic analytic torsion
\cite{BismutGilletSoule:at}, \cite{BismutGilletSoule:atII},
\cite{BismutGilletSoule:atIII} and deep results of Bismut-Lebeau
\cite{BismutLebeau:CiQm} on the compatibility of analytic torsion with
closed immersions. The holomorphic analytic torsion was later
generalized by Bismut-K\"ohler \cite{BismutKohler}, to the holomorphic
analytic torsion forms, that transgress the whole
Grothendieck-Riemann-Roch theorem for K\"ahler submersions, at the
level of differential forms. The extension of the arithmetic
Grothendieck-Riemann-Roch theorem to the full Chern character and
generically smooth morphisms was finally proven by
Gillet-R\"ossler-Soul\'e
\cite{GilletRoesslerSoule:_arith_rieman_roch_theor_in_higher_degrees}. They
applied the analogue to the Bismut-Lebeau immersion theorem, for
analytic torsion forms, established in the monograph by Bismut
\cite{Bismut:Asterisque}.

Theorem \ref{thm:5} provides an extension of the previous results in
several directions. First, we allow the morphism to be an arbitrary
projective morphism, non necessarily generically smooth. Second, we
can deal with metrized objects in the bounded derived category of
coherent sheaves. Third, we provide all the possible forms of such a
theorem, by introducing our theory of generalized analytic torsion
classes, thus explaining  the topological correction term. 

\paragraph{Reductions.} The proof of our version of the arithmetic
Grothendieck-Riemann-Roch theorem follows the pattern of the classical
approach in algebraic geometry. Namely, it proceeds by factorization
of the morphism $\ov{f}$ into a regular closed immersion and a trivial
projective bundle projection. The advantage of working with the
formalism of hermitian structures on the derived category of coherent
sheaves makes the whole procedure more transparent and direct, in
particular avoiding the appearance of several secondary classes. Also
the cocyle type relation expressing the behaviour of generalized
analytic torsion with respect to composition of morphisms in
$\ov{\Sm}_{\ast/\RR}$ (see the axioms
\cite[Sec. 7]{BurgosFreixasLitcanu:GenAnTor}) is well suited to this
factorization argument. 

By the classification of generalized analytic torsion theories, it is
enough to prove Theorem \ref{thm:5} for the homogenous theory
$T^{h}$. From now on we fix this choice, and hence all derived direct
image functors will be with respect to this theory. 

Let $\ov{f}$ be a morphism in $\ov{\Reg}_{\ast/A}$, and consider a factorization
\begin{displaymath}
  \xymatrix{
    \mathcal{X}\ar@{^{(}->}[r]^{i}\ar[rd]^{f}
    &\mathcal{\PP}^{n}_{\mathcal{Y}}\ar[d]^{\pi}\\
    &\mathcal{Y}.
  }
\end{displaymath}
The tangent complex of $\pi$ is canonically isomorphic to $p^{\prime\ast}T_{\PP^{n}_{A}/A}$, where 
\begin{displaymath}
	\xymatrix{
		\PP^{n}_{\mathcal{Y}}\ar[r]^-{p^{\prime}}\ar[d]_{\pi}	&\PP^{n}_{A}\ar[d]^{\pi^{\prime}}\\
		\mathcal{Y}\ar[r]_-{p}	&\Spec A.
	}
\end{displaymath}
We may thus endow $\pi$ with the pull-back by $p^{\prime}$ of the Fubini-Study metric on $T_{\PP^{n}_{A}/A}$ \cite[Sec. 5]{BurgosFreixasLitcanu:GenAnTor}. We write $\ov{\pi}$ for the resulting morphism in $\ov{\Reg}_{\ast/A}$. By \cite[Lemma 5.3]{BurgosFreixasLitcanu:HerStruc}, there exists a unique hermitian structure on $i$ such that $\ov{f}=\ov{\pi}\circ\ov{i}$. Then we recall that we have the equality of functors
\begin{displaymath}
	f_{\ast}=\ov{\pi}_{\ast}\circ\ov{i}_{\ast}
\end{displaymath}
and the equality of arithmetic Todd genera
\begin{displaymath}
	\widehat{\Td}(\ov{f})=i^{\ast}\widehat{\Td}(\ov{\pi})\widehat{\Td}(\ov{i}).
\end{displaymath}
\begin{lemma}
It is enough to prove Theorem \ref{thm:5} for $\ov{i}$ and $\ov{\pi}$ individually.
\end{lemma}
\begin{proof}
Let us assume the theorem known for $\ov{i}$ and for $\ov{\pi}$. Because the theory is known for $\ov{\pi}$, we may apply it to the object $\ov{i}_{\ast}\alpha$ in $\oDb(\PP^{n}_{\mathcal{Y}})$, to obtain
\begin{equation}\label{eq:30}
	\chh(\ov{f}_{\ast}\alpha)=\chh(\ov{\pi}_{\ast}(\ov{i}_{\ast}\alpha))=
		\pi_{\ast}(\chh(\ov{i}_{\ast}\alpha)\widehat{\Td}(\ov{\pi})).
\end{equation}
Because the theorem is known for $\ov{i}$, we also have
\begin{equation}\label{eq:31}
	\begin{split}
		\chh(\ov{i}_{\ast}\alpha)\widehat{\Td}(\ov{\pi})=&i_{\ast}(\chh(\alpha)\widehat{\Td}(\ov{i}))\widehat{\Td}(\ov{\pi})\\
		&=i_{\ast}(\chh(\alpha)\widehat{\Td}(\ov{i})i^{\ast}\widehat{\Td}(\ov{\pi}))\\
		&\hspace{0.3cm}=i_{\ast}(\chh(\alpha)\widehat{\Td}(\ov{f})).
	\end{split}
\end{equation}
Observe we used the projection formula in arithmetic Chow groups
(Theorem \ref{thm:2}), and that this equality holds in
$\bigoplus_{p}\cha^{p}_{\QQ}(\mathcal{Y},\mathcal{D}_{\D,\as}(Y,f_{\ast}(W)))$,
because the definition of direct image of wave front sets and the fact
that $\pi$ is smooth yield $f_{\ast}(W)=\pi_{\ast}(i_{\ast}(W))$. We
conclude by putting \eqref{eq:30}--\eqref{eq:31} together and using
the fact $f_{\ast}=\pi_{\ast} i_{\ast}$ on arithmetic Chow groups
(Proposition \ref{prop:8}).
\end{proof}
\begin{lemma} \label{lemm:4}
Theorem \ref{thm:5} holds for $\ov{i}$.
\end{lemma}
\begin{proof}
In \cite[Thm 10.28]{BurgosLitcanu:SingularBC}, the authors prove the
theorem for direct images by closed immersions, defined on arithmetic
$K$ groups. They suppose as well that the hermitian structure on
$T_{i}$ is given by a smooth hermitian metric on
$N_{\mathcal{X}/\PP^{n}_{\mathcal{Y}}}$. Finally, the result in
\emph{loc,. cit.} holds in
$\bigoplus_{p}\cha^{p}_{\QQ}(\mathcal{Y},\mathcal{D}_{\D, \as}(Y))$. 

By the anomaly formula provided by Theorem \ref{thm:4} and the anomaly
formula of generalized analytic torsion theories
\cite[Prop. 7.4]{BurgosFreixasLitcanu:GenAnTor} for change of
hermitian structure on the tangent complex, one can extend \cite[Thm
10.28]{BurgosLitcanu:SingularBC} to arbitrary hermitian structures on
$i$, in particular the one we fixed. Also, a careful proof reading of
the proof in \emph{loc. cit.} shows that the theorem can be adapted to
allow
$\alpha\in\widehat{K}_{0}(\mathcal{X},\mathcal{D}_{\D,\as}(X,W))$,
taking values in $\bigoplus_{p}\cha^{p}_{\QQ}(\PP^{n}_{Y},\mathcal{D}_{\D,
  \as}(\PP^{n}_{Y},i_{\ast}W))$. Finally, to get the result for
$\ov{i}_{\ast}$ on the arithmetic derived category, we use the class
functor to arithmetic $K$-theory, the commutativity Theorem
\ref{thm:6} and that the Chern character on arithmetic derived
category factors, by construction, through arithmetic $K$ groups (see
Notation \ref{not:1}).
\end{proof}

\begin{lemma}
To prove Theorem \ref{thm:5} for $\ov{\pi}$, it is enough to prove it
when $\mathcal{Y}=\Spec A$ and $\alpha=\ov{\OO(k)}$, $-n\leq k\leq 0$,
where the chosen hermitian structure on $\OO(k)$ is the Fubini-Study
metric. 
\end{lemma}
\begin{proof}
The derived category $\Db(\PP^{n}_{\mathcal{Y}})$ is generated by
coherent sheaves of the form $\pi^{\ast}\mathcal{G}\otimes
p^{\prime\ast}\OO(k)$
\cite[Cor. 4.11]{BurgosFreixasLitcanu:GenAnTor}. By the anomaly
formulas \cite[Prop. 7.4, Prop. 7.6]{BurgosFreixasLitcanu:GenAnTor},
we thus reduce to prove the theorem for $\alpha$ of the form
$\pi^{\ast}\ov{\mathcal{G}}\otimes p^{\prime\ast}\ov{\OO(k)}$, for any
metric on $\mathcal{G}$. On the one hand, by the formulas in Prop
\ref{prop:5}, the multiplicativity and pull-back functoriality
(Prop. \ref{prop:6}) of the Chern character, we have 
\begin{displaymath}
  \begin{split}
    \chh(\ov{\pi}_{\ast}\alpha)=&\chh(\ov{\mathcal{G}}\otimes\ov{\pi}_{\ast}p^{\prime\ast}\ov{\OO(k)})\\
    =&\chh(\ov{\mathcal{G}})\chh(\ov{\pi}_{\ast}p^{\prime\ast}\ov{\OO(k)})\\
    &\hspace{0.3cm}=\chh(\ov{\mathcal{G}})\chh(p^{\ast}\ov{\pi}^{\prime}_{\ast}\ov{\OO(k)})\\
    &\hspace{0.6cm}=\chh(\ov{\mathcal{G}})p^{\ast}\chh(\ov{\pi}^{\prime}_{\ast}\ov{\OO(k)}).
  \end{split}
\end{displaymath}
Here we recall $\pi^{\prime}:\PP^{n}_{A}\rightarrow\Spec A$ is the
structure morphism, that we endow with the Fubini-Study metric. On the
other hand, we similarly prove
\begin{displaymath}
  \begin{split}
    \pi_{\ast}(\chh(\pi^{\ast}\ov{\mathcal{G}}\otimes p^{\prime\ast}\ov{\OO(k)})\widehat{\Td}(\ov{\pi}))
    =&\chh(\ov{\mathcal{G}})\pi_{\ast}(p^{\prime\ast}\chh(\ov{\OO(k)})p^{\prime\ast}\widehat{\Td}(\ov{\pi}^{\prime}))\\
    &=\chh(\ov{\mathcal{G}})p^{\ast}\pi^{\prime}_{\ast}(\chh(\ov{\OO(k)})\widehat{\Td}(\ov{\pi^{\prime}})).
  \end{split}
\end{displaymath}
We thus reduce to prove the theorem for $\ov{\pi}^{\prime}$ and
$\alpha=\ov{\OO(k)}$, as was to be shown.
\end{proof}
\paragraph{The case of projective spaces.} To prove Theorem
\ref{thm:5} in full generality, it remains to treat the case of the
projection $\ov{\pi}\colon\PP^{n}_{A}\rightarrow\Spec A$, where we
endow $\pi$ with the Fubini-Study metric. Since this projection is the
pull-back of the projection
$\ov{\pi}\colon\PP^{n}_{\ZZ}\rightarrow\Spec \ZZ$ it is enough to
treat the case when $A=\ZZ$. 

Furthermore, we showed that
it is enough to consider $\alpha$ of the form $\ov{\OO(k)}$, $-n\leq
k\leq 0$, with the Fubini-Study metric as well and any metric of the
direct image $\pi_{\ast}\OO(k)$.
In \cite[Def. 5.7]{BurgosFreixasLitcanu:GenAnTor}, we introduced the
main characteristic numbers of $T^{h}$,
\begin{displaymath}
  t_{n,k}^{h}=T^{h}(\ov{\pi},\ov{\OO(k)},\ov{\pi_{\ast}\OO(k)}),\
  -n\le k \le 0,
\end{displaymath}
where $\pi_{\ast}\OO(k)$ was endowed with its $L^{2}$ metric. Since we
will only consider the homogeneous analytic torsion we will shorthand
$t^{h}_{n,k}=t_{n,k}$. 

 We will
denote by $\ov 0$ the trivial vector bundle with trivial hermitian
structure and, for any $X$, we denote by $\ov \OO_{X}$ the structural
sheaf with the metric $\|1\|=1$.
For
$-n\le k < 0$, the complex $\ov{\pi_{\ast}\OO(k)}$ can be represented
by $\ov 0$, while 
the complex $\ov{\pi_{\ast}\OO(0)}$ can be represented by the
structural sheaf $\OO_{\ZZ}$ with the metric $\|1\|=1/n!$. Since this
metric depends on $n$ it will be simpler to consider the complexes
$\ov{\pi_{\ast}\OO(k)}'=\ov{\pi_{\ast}\OO(k)}$ for $-n\le k < 0$ and 
$\ov{\pi_{\ast}\OO(k)}'=\ov \OO_{\ZZ}$.
Then we write 
\begin{displaymath}
  t'_{n,k}=T^{h}(\ov{\pi},\ov{\OO(k)},\ov{\pi_{\ast}\OO(k)}'),\
  -n\le k \le 0.
\end{displaymath}
These characteristic numbers satify
\begin{displaymath}
  t'_{n,k}=
  \begin{cases}
    t_{n,k},&\text{ if }-n\le k<0,\\
    t_{n,0}-(1/2)\log(n!),& \text{ if }k=0.
  \end{cases}
\end{displaymath} 
Clearly it is equivalent to work with the characteristic numbers
$t_{n,k}$ or with $t'_{n,k}$.

In order to finish the the proof of Theorem \ref{thm:5} it only
remains to show the following particular cases.

\begin{theorem}\label{thm:7}
For every $0\leq k\leq n$, we have the equality in $\cha^{1}(\Spec
\ZZ)=\RR$ 
\begin{equation}\label{eq:32}
  \amap(t'_{n,-k})=\chh(\ov{\pi_{\ast}\OO(-k)}')-\
  \pi_{\ast}(\chh(\ov{\OO(-k)})\widehat{\Td}(\ov{\pi})).
\end{equation}
\end{theorem}
The proof will proceed by induction. The next proposition treats the first case.
\begin{proposition}\label{prop:7}
Equation \eqref{eq:32} holds for $k=0$.
\end{proposition}
\begin{proof}
The proof exploits the behavior of generalized analytic torsion with
respect to composition of morphisms. Let us consider the diagram
\begin{displaymath}
  \xymatrix{
    \PP^n_{\ZZ} \ar[rd]_{\Id} \ar[r]^{\Delta\ \ \ }   & \PP^n_{\ZZ}
    \underset{\ZZ}{\times}\PP^n_{\ZZ} \ar[r]^{p_1}\ar[d]^{p_2} &
    \PP^n_{\ZZ} \ar[d]^{\pi}\\ 
    & \PP^n_{\ZZ} \ar[r]_{\pi_{1}} & \Spec \ZZ .}
\end{displaymath}
The Fubini-Studi metric on the tangent space $T_{\PP^{n}_{\CC}}$
is invariant under complex conjugation. It induces a 
metric on the tangent space of the product of projective spaces. On
each morphism on the above diagram we consider the relative hermitian
structure deduced by the hermitian metrics on each tangent space. With
this choice 
\begin{equation}
  \label{eq:24}
  \ov \Id=\ov p_{2}\circ \ov {\Delta },
\end{equation}
where the composition of relative hermitian structures is defined in
\cite[Definition 5.7]{BurgosFreixasLitcanu:HerStruc}. On the relative
tangent bundle  $T_{p_{2}}$
we consider the metric induced by the metric on $T_{\PP^{n}\times
  \PP^{n}}$.
Since the
short exact sequence
\begin{displaymath}
  0\longrightarrow
  \ov T_{p_{2}}\longrightarrow
  \ov T_{\PP^{n}\times \PP^{n}}\longrightarrow
  p_{2}^{\ast}\ov T_{\PP^{2}} \longrightarrow 0
\end{displaymath}
is orthogonaly split,
we deduce that the metric we consider on $\ov p_{2}$ agrees with the
metric on the vector bundle
$\ov T_{p_{2}}$ and this, in turn agrees with the metric on
$p_{1}^{\ast }T_{\PP^{n}}$. 

Let $\ov Q$ be the tautological quotient bundle on $\PP^{n}_{\ZZ}$ with
any hermitian metric invariant under complex conjugation. We denote by
$K$ the Koszul resolution of the diagonal and $\ov K$ the same
resolution with the induced metrics. Namely $\ov K$ is the complex
\begin{displaymath}
  0\to p_{2}^{\ast}\Lambda ^{n}\ov Q^{\vee}\otimes p_{1}^{\ast}\ov
  {\OO(-n)}
\to \dots\to p_{2}^{\ast}\ov Q^{\vee}\otimes p_{1}^{\ast}\ov
  {\OO(-1)}\to \ov \OO_{\PP^{n}\times\PP^{n}}\to 0,
\end{displaymath}
and there is an isomorphism in the derived category $K\to \Delta
_{\ast}\OO_{\PP^{n}_{\ZZ}}$. Since the theories of homogeneous torsion classes
for closed immersions and for projective spaces are compatible
\cite[Definition 6.2]{BurgosFreixasLitcanu:GenAnTor} the equation
\begin{equation}
  \label{eq:25}
  T(\ov p_{2},\ov K, \ov \OO_{\PP^{n}})+ \ov p_{2\ttwist}(T(\ov \Delta
  ,\ov \OO_{\PP^{n}},\ov K))=0
\end{equation}
holds. Then
\begin{align*}
  T(\ov p_{2},\ov K, \ov \OO_{\PP^{n}})&=
  T(\ov p_{2},\ov \OO_{\PP^{n}\times \PP^{n}},\ov \OO_{\PP^{n}})+
  \sum_{i=1}^{n}(-1)^{i}T(\ov p_{2}, p_{2}^{\ast}\Lambda ^{i}\ov
  Q^{\vee}\otimes p_{1}^{\ast}\ov
  {\OO(-i)},\ov 0)\\
  &=\pi_{1}^{\ast}T(\ov \pi ,\ov \OO_{ \PP^{n}},\ov
  \OO_{\ZZ})\bullet \ch(\ov \OO_{\PP^{n}})+\\
  &\qquad \qquad \sum_{i=1}^{n}(-1)^{i}\pi_{1}^{\ast}T(\ov \pi ,\ov {\OO
    (-i)},\ov 0)\bullet \ch(\Lambda ^{i}\ov Q^{\vee}) \\
  &=\sum_{i= 0}^{n}(-1)^{i}t_{n,-i}\bullet \ch(\Lambda ^{i}\ov
  Q^{\vee}). 
\end{align*}

Using that
\begin{displaymath}
  \pi _{1\ast}(\ch(\Lambda ^{i}\ov
  Q^{\vee})\Td(\ov \pi _{1}))=
  \begin{cases}
    1, &\text{ if }i=0,\\
    0, &\text{ otherwise,} 
  \end{cases}
\end{displaymath}
we deduce
\begin{equation}
  \label{eq:33}
  \ov \pi _{1\ttwist}(T(\ov p_{2},\ov K, \ov \OO_{\PP^{n}}))=t_{n,0}.
\end{equation}

By the arithmetic Riemann-Roch theorem for closed immersions
\begin{displaymath}
  \amap(T(\ov \Delta
  ,\ov \OO_{\PP^{n}},\ov
  K))=
  \sum_{i=0}^{n}(-1)^{i}p_{2}^{\ast}\chh(\Lambda ^{i}\ov
  Q^{\vee})\cdot
  p_{1}^{\ast}\chh(\ov{\OO(-i)})-\Delta _{\ast}(\chh(\ov
  \OO_{\PP^{n}})\widehat{\Td}(\ov \Delta )).
\end{displaymath}
For $i>0$
\begin{multline*}
  \pi _{1\ast}\left(p_{2\ast}\left(p_{2}^{\ast}\chh(\Lambda ^{i}\ov
  Q^{\vee})p_{1}^{\ast}\chh(\ov {\OO(-i)})\widehat {\Td}(\ov{p_{2}})\right)
  \widehat {\Td}(\ov{\pi} _{1})\right)\\
  =\pi _{1\ast}(\chh(\Lambda ^{i}\ov
  Q^{\vee})\widehat {\Td}(\ov{\pi} _{1}))\cdot \pi _{\ast}(\chh(\ov
  {\OO(-i)})\widehat{\Td}(\ov \pi )).
\end{multline*}
Since
\begin{displaymath}
  \zeta (\pi _{1\ast}(\chh(\Lambda ^{i}\ov
  Q^{\vee})\widehat {\Td}(\ov \pi _{1})))=
 \zeta(\pi _{\ast}(\chh(\ov
  {\OO(-i)})\widehat{\Td}(\ov \pi )))=0
\end{displaymath}
and $\Ker \zeta $ is a square zero ideal of the arithmetic Chow ring,
we deduce that, for $i>0$
\begin{equation}\label{eq:36}
  \pi _{1\ast}\left(p_{2\ast}\left(p_{2}^{\ast}\chh(\Lambda ^{i}\ov
  Q^{\vee})p_{1}^{\ast}\chh(\ov {\OO(-i)})\widehat {\Td}(\ov{p_{2}})\right)
  \widehat {\Td}(\ov \pi _{1})\right)=0.
\end{equation}
For $i=0$ we compute
\begin{displaymath}
  \pi _{1\ast}\left(p_{2\ast}\left(p_{2}^{\ast}\chh(\ov
  \OO_{\PP^{n}})p_{1}^{\ast}\chh(\ov \OO_{\PP^{n}})\widehat {\Td}(\ov{p_{2}})\right)
  \widehat {\Td}(\ov \pi _{1})\right)=\pi _{\ast}(\chh(\ov
\OO_{\PP^{n}})\widehat{\Td}(\ov \pi))^{2}.
\end{displaymath}
Using that $\chh(\ov \OO_{\PP^{n}})=1$ and that $\pi
_{\ast}(\widehat{\Td}(\ov \pi))-1\in \Ker \zeta $ we obtain that
\begin{equation}\label{eq:35}
  \pi _{\ast}(\chh(\ov
\OO_{\PP^{n}})\widehat{\Td}(\ov \pi))^{2}=-1+2\pi
_{\ast}(\widehat{\Td}(\ov \pi )).
\end{equation}
Furthermore, by the choice of metrics on the relative tangent complexes
\begin{equation}\label{eq:34}
  \pi _{1\ast}p_{2\ast}\Delta _{\ast}(\chh(\ov
  \OO_{\PP^{n}})\widehat{\Td}(\ov \Delta )\widehat{\Td}(\ov
  p_{2})\widehat{\Td}(\ov\pi _{1}))=
\pi _{\ast}(\chh(\ov
  \OO_{\PP^{n}})\widehat{\Td}(\ov\pi )).
\end{equation}
Using equations \eqref{eq:36}, \eqref{eq:35} and \eqref{eq:34} we
deduce that
\begin{equation}
  \label{eq:37}
  \amap(\ov\pi _{1\ttwist}(\ov p_{2\ttwist}(T(\ov \Delta ,\ov
  \OO_{\PP^{n}},\ov K))))=\pi _{\ast}(\chh(\ov
  \OO_{\PP^{n}})\widehat{\Td}(\ov\pi ))-\chh(\ov \OO_{\ZZ}).
\end{equation}
By equations \eqref{eq:25}, \eqref{eq:33} and \eqref{eq:37} we conclude
\begin{displaymath}
  \amap(t'_{n,0})=\chh(\ov \OO_{\ZZ})-\pi _{\ast}(\chh(\ov
  \OO_{\PP^{n}})\widehat{\Td}(\ov\pi ))
\end{displaymath}
proving the proposition.
\end{proof}
\begin{proof}[Proof of Theorem \ref{thm:7}] We now proceed with the
  induction step. We assume that equation \eqref{eq:32} holds for some
  $k\ge 0$ and all $n\ge k$. Fix now $n\ge k+1$. Consider the diagram 
\begin{displaymath}
  \xymatrix{
    \PP^{n-1}_{\ZZ} \ar[rd]_{\pi _{n-1}} \ar[r]^{s}   & \PP^n_{\ZZ}
    \ar[d]^{\pi _{n}}\\ 
    & \Spec \ZZ,}
\end{displaymath}
  where $s$ is the closed immersion induced by a section $s$ of
  $\OO_{\PP^{n}}(1)$. As in the proof of Proposition \ref{prop:7}, we
  consider the relative hermitian 
  structures defined by the Fubini-Study metric on the tangent
  bundles. In this way $\ov \pi _{n-1}=\ov \pi _{n}\circ \ov s$.

  We consider the Koszul complex
  \begin{displaymath}
    K_{n,k}\colon \OO_{\PP^{n}}(-k-1)\overset{s}{\longrightarrow }
    \OO_{\PP^{n}}(-k)
  \end{displaymath}
  and we denote by $\ov K_{n,k}$ the same complex provided with the
  Fubini-Study metrics.

  By the transitivity \cite[Definition
  7.1]{BurgosFreixasLitcanu:GenAnTor} of the homogeneous theory
  of analytic torsion, the equation
  \begin{multline*}
    T^{h}(\ov \pi _{n-1},\ov {\OO(-k)},\ov {\pi _{n-1\ast}\OO(-k)}')=\\
    T^{h}(\ov \pi _{n},\ov K_{n,k},\ov {\pi _{n-1\ast}\OO(-k)}')+
    \ov\pi _{n\ttwist}(T^{h}(\ov s,\ov {\OO(-k)},\ov K_{n,k})).
  \end{multline*}
  By definition $T^{h}(\ov \pi _{n-1},\ov {\OO(-k)},\ov {\pi
    _{n-1\ast}\OO(-k)}')=t'_{n-1,-k}$. By the choice of metrics, the
  exact sequence
  \begin{equation} \label{eq:38}
    0\to \ov {\pi _{n\ast}\OO(-k-1)}'
    \to \ov {\pi _{n\ast}\OO(-k)}' \to \ov {\pi
      _{n-1\ast}\OO(-k)}'\to 0
  \end{equation}
  is orthogonal split (all the terms appearing in this short exact
  sequence are either $\ov 0$ or $\ov
  \OO_{\ZZ}$). 
  This implies that
  \begin{multline*}
    T^{h}(\ov \pi _{n},\ov K_{n,k},\ov {\pi _{n-1\ast}\OO(-k)}')=\\
    T^{h}(\ov \pi _{n},\ov{\OO(-k)},\ov {\pi _{n\ast}\OO(-k)}')
    -T^{h}(\ov \pi _{n},\ov{\OO(-k-1)},\ov {\pi
      _{n\ast}\OO(-k-1)}')=\\
    t'_{n,-k}-t'_{n,-k-1}.
  \end{multline*}
  By Lemma \ref{lemm:4} (the case of closed immersions)
  \begin{displaymath}
    T^{h}(\ov s,\ov {\OO(-k)},\ov K_{n,k})=\chh(\ov {\OO(-k)})-
    \chh(\ov{\OO}(-k-1))-\ov{s}_{\ttwist}(\chh(\ov{\OO(-k)})).
  \end{displaymath}
This implies that
\begin{multline*}
  \ov{\pi} _{n\ttwist}(T^{h}(\ov s,\ov {\OO(-k)},\ov K_{n,k}))=\\ \ov{\pi}
  _{n\ttwist}\chh(\ov {\OO(-k)})- 
    \ov{\pi} _{n\ttwist} \chh(\ov{\OO}(-k-1))
    -\ov{\pi} _{n-1\ttwist}\chh(\ov{\OO(-k)}).
\end{multline*}
Summing up, we deduce
\begin{multline}\label{eq:39}
  \amap(t'_{n-1,-k}-t'_{n,-k}+t'_{n,-k-1})=\\
  -\ov{\pi} _{n-1\ttwist}\chh(\ov{\OO(-k)})+
  \ov{\pi} _{n\ttwist}\chh(\ov{\OO(-k)})+
  -\ov{\pi} _{n\ttwist}\chh(\ov{\OO(-k-1)}).
\end{multline}
Applying the induction hypothesis we get
\begin{displaymath}
  t'_{n,-k-1}=-\chh(\ov{\pi _{n-1\ast}\OO(-k)}')+
  \chh(\ov{\pi _{n\ast}\OO(-k)}')-\ov{\pi} _{n\ttwist}\chh(\ov{\OO(-k-1)}). 
\end{displaymath}
Using again that the exact sequence \eqref{eq:38} is orthogonally
split, we deduce that
\begin{displaymath}
  t'_{n,-k-1}=\chh(\ov{\pi _{n\ast}\OO(-k-1)}')-\ov{\pi} _{n\ttwist}\chh(\ov{\OO(-k-1)}). 
\end{displaymath}
completing the inductive step and proving the Theorem \ref{thm:7} and
therefore Theorem \ref{thm:5}.
\end{proof}

Once we have proved that the formal properties of an analytic torsion
theory imply the arithmetic Riemann-Roch theorem we can compute
easily the characteristic numbers $t_{n,k}$ for the homogeneous
analytic torsion and all $n\ge 0$ and $k\in \ZZ$. By the Theorem
\ref{thm:5} they satisfy
\begin{displaymath}
  \amap(t_{n,k})=\chh(\ov{\pi_{n\ast}\OO_{\PP^{n}_{\ZZ}}(k)})-\
  \pi_{n\ast}(\chh(\ov{\OO_{\PP^{n}_{\ZZ}}(k)})\widehat{\Td}(\ov{\pi}_{n})).
\end{displaymath}
Observe it is enough to compute $t_{n,k}$ for $k\geq -n$, since the self-duality of the homogenous analytic torsion \cite[Thm. 9.12]{BurgosFreixasLitcanu:GenAnTor} immediately yields the relation
\begin{displaymath}
	t_{n,k}=(-1)^{n}t_{n,-k-n-1}.
\end{displaymath}
Therefore, from now on we restrict to this range of values of $k$.

We first compute
$\chh(\ov{\pi_{n\ast}\OO_{\PP^{n}_{\ZZ}}(k)})^{(1)}$. For $-n\leq k\leq -1$ this quantity vanishes:
\begin{displaymath}
	 \chh(\ov{\pi_{n\ast}\OO_{\PP^{n}_{\ZZ}}(k)})^{(1)}=0,\quad -n\leq k\leq -1.
\end{displaymath}
Suppose now $k\geq 0$. Using that the volume
form $1/n!\omega _{FS}^{n}$ is given, in a coordinate patch, by
\begin{displaymath}
  \mu =\left(\frac{i}{2\pi }\right)^{n}\frac{\dd z_{1}\land \dd \ov
    z_{1}\land\dots \land \dd z_{n}\land \dd \ov
    z_{n}}{(1+\sum_{i=1}^{n}z_{i}\ov z_{i})^{n+1}},
\end{displaymath}
it is easy to see that the basis $\{x_{0}^{a_{0}}\dots
x_{n}^{a_{n}}\}_{a_{0}+\dots+a_{n}=k}$ is orthonormal and satisfies
\begin{displaymath}
  \|x_{0}^{a_{0}}\dots
x_{n}^{a_{n}}\|^{2}_{L^{2}}=\frac{a_{0}!\dots a_{n}!}{(k+n)!}.
\end{displaymath}
Therefore
\begin{displaymath}
  \chh(\ov{\pi_{n\ast}\OO_{\PP^{n}_{\ZZ}}(k)})^{(1)}=
  \sum_{a_{0}+\dots+a_{n}=k}-\left(\frac{1}{2}\right)
  \log\left(\frac{a_{0}!\dots a_{n}!}{(k+n)!}\right).
\end{displaymath}

To compute
$\pi_{n\ast}(\chh(\ov{\OO_{\PP^{n}_{\ZZ}}(k)})\widehat{\Td}(\ov{\pi}_{n}))$
we follow \cite{GilletSoule:ATAT}, where the case $k=0$ is
considered. Let $\alpha _{n,k}$ be the coefficient of $x^{n+1}$
in the power series
\begin{math}
e^{kx}\left(\frac{x}{1-e^{-x}}\right)^{n+1}  
\end{math}
and let $\beta _{n,k}$ be the coefficient of $x^{n}$ in the power series
\begin{math}
  \int_{0}^{1}\frac{\phi(t)-\phi(0)}{t}\dd t, 
\end{math}
where
\begin{displaymath}
  \phi (t)=e^{kx}\left(\frac{1}{tx}-\frac{e^{-tx}}{1-e^{-tx}}\right)
  \left(\frac{x}{1-e^{-x}}\right)^{n+1}.
\end{displaymath}
Then, by a slight modification of the argument in \cite [Proposition
2.2.2 \& 2.2.3]{GilletSoule:ATAT}, we derive
\begin{displaymath}
  \pi
  _{n\ast}(\chh(\ov{\OO_{\PP^{n}_{\ZZ}}(k)})
  \widehat{\Td}(\ov{\pi}_{n}))^{(1)}=\frac{1}{2}   
  \amap\left(\alpha _{n,k}\sum_{p=1}^{n}\sum_{j=1}^{p}\frac{1}{j}+\beta
    _{n,k}\right).
\end{displaymath}
The factor $(1/2)$ appears from the different normalization used here (see
\cite[Theorem 3.33]{BurgosKramerKuehn:accavb}). We thus have to determine the coefficients $\alpha_{n,k}$ and $\beta_{n,k}$. The numbers $\alpha_{n,k}$ can be obtained similarly to \cite[Eq. (27)]{GilletSoule:ATAT}. We obtain
\begin{displaymath}
	\alpha_{n,k}=\begin{cases}
		0	&\text{for } -n\leq k\leq -1,\\
		\binom{k+n}{n}	&\text{for }k\geq 0.
	\end{cases}
\end{displaymath}
The values of $\beta_{n,k}$ are expressed in terms of some secondary Todd numbers, that in turn are determined by a generating series:
\begin{displaymath}
	\beta_{n,k}=\sum_{j=0}^{n}\widetilde{\Td}_{n-j}\frac{k^{j}}{j!},
\end{displaymath}
where the $\widetilde{\Td}_{m}$ are given by the equality of generating series
\begin{equation}\label{eq:todd_numbers}
	\sum_{m\geq 0}\frac{\widetilde{\Td}_{m}}{m+1}T^{m+1}=
	\sum_{m\geq 1}\frac{\zeta(-(2m-1))}{2m-1}\frac{y^{2m}}{(2m)!},\quad T=1-e^{-y}.
\end{equation}
Here $\zeta$ stands for the Riemann zeta function. The result is a direct consequence of the definition of $\beta_{n,k}$ and the computation of the numbers $\beta_{n,0}$ in \cite[Prop. 2.2.3 \& Lemma 2.4.3]{GilletSoule:ATAT}. 

We summarize these computations for the principal characteristic numbers $t_{n,k}$, $-n\leq k\leq 0$, since it is particularly pleasant.
\begin{proposition}
The principal characteristic numbers $t_{n,k}$ are given by
\begin{displaymath}
	t_{n,k}=\begin{cases}
		-\frac{1}{2}\sum_{p=1}^{n}\sum_{j=1}^{p}\frac{1}{j},	&\text{for }k=0,\\	
		\\
	-\frac{1}{2}\sum_{j=0}^{n}\widetilde{\Td}_{n-j}\frac{k^{j}}{j!},		&\text{for } -n\leq k\leq -1,
\end{cases}
\end{displaymath}
where the sequence of numbers $\widetilde{\Td}_{m}$, $m\geq 0$, is determined by \eqref{eq:todd_numbers}. 
\end{proposition}

\newcommand{\noopsort}[1]{} \newcommand{\printfirst}[2]{#1}
  \newcommand{\singleletter}[1]{#1} \newcommand{\switchargs}[2]{#2#1}
  \def\cprime{$'$}
\providecommand{\bysame}{\leavevmode\hbox to3em{\hrulefill}\thinspace}
\providecommand{\MR}{\relax\ifhmode\unskip\space\fi MR }
\providecommand{\MRhref}[2]{%
  \href{http://www.ams.org/mathscinet-getitem?mr=#1}{#2}
}
\providecommand{\href}[2]{#2}


\end{document}